\documentclass[12pt]{amsart}

\usepackage{amsaddr}
\usepackage{amsmath}
\usepackage{amssymb}
\usepackage{amscd}
\usepackage{latexsym}
\usepackage{amsthm}
\usepackage{mathrsfs}
\usepackage{color}
\usepackage[backend=bibtex,style=alphabetic,isbn=false,url=false,maxnames=5,firstinits=true]{biblatex}
\usepackage{amsfonts}
\usepackage{enumitem}
\usepackage{array}
\usepackage{setspace}
\usepackage{tikz}
\usepackage{tikz-cd}
\usetikzlibrary{arrows}
\renewbibmacro{in:}{}

\usepackage{stmaryrd,graphicx}

\newtheorem{thm}{Theorem}[section]
\newtheorem*{thm*}{Theorem}
\newtheorem{cor}[thm]{Corollary}
\newtheorem*{cor*}{Corollary}
\newtheorem{lem}[thm]{Lemma}
\newtheorem*{lem*}{Lemma}
\newtheorem{prop}[thm]{Proposition}
\newtheorem*{prop*}{Proposition}

\theoremstyle{definition}
\newtheorem{defn}{Definition}[section]
\newtheorem*{defn*}{Definition}
\newtheorem*{conjecture}{Conjecture}

\theoremstyle{remark}
\newtheorem{rem}{Remark}[section]
\newtheorem*{rem*}{Remark}
\newtheorem{example}{Example}[section]

\newtheorem*{problem*}{Problem}

\newcommand{\Q}{\mathbb Q}

\newcommand{\C}{\mathbb C}

\newcommand{\BA}{\mathbf A}
\newcommand{\BB}{\mathbf B}

\newcommand{\BZ}{\mathbf Z}

\newcommand{\Z}{\mathbb Z}

\newcommand{\cP}{\mathcal P}
\newcommand{\cS}{\mathcal S}
\newcommand{\cT}{\mathcal T}

\newcommand{\HH}{\mathbb H}

\DeclareMathOperator{\pExp}{Exp}
\DeclareMathOperator{\pLog}{Log}
\DeclareMathOperator{\Aut}{Aut}
\DeclareMathOperator{\GL}{GL}

\DeclareMathOperator{\opflat}{flat}
\DeclareMathOperator{\Tr}{Tr}

\DeclareMathOperator{\Sym}{Sym}

\title{Integrality of HLV kernels}
\author{Anton Mellit}
\email{anton.mellit@univie.ac.at}
\address{Faculty of Mathematics, University of Vienna, \\
	Oskar-Morgenstern-Platz 1, 1090 Vienna, Austria}
\date{\today}
\begin{document}

\onehalfspacing

\begin{abstract}
We prove that the coefficients of the generating function of Hausel, Letellier, Villegas, and its recent generalization by Carlsson and Villegas, which according to various conjectures should compute mixed Hodge numbers of character varieties and moduli spaces of Higgs bundles of curves of genus $g$ with $n$ punctures, are polynomials in $q$ and $t$ with integer coefficients for any $g,n\geq 0$.
\end{abstract}

\maketitle

\section{Introduction}
Consider a Riemann surface of genus $g$ with $n$ punctures, $g,n\geq 0$.
To formulate the main result and to fix the notations we introduce the \emph{exponential HLV kernel} as follows:
\begin{equation}\label{eq:omegaintro}
\Omega_{u_1,u_2,\ldots, u_g}[X_1,X_2,\ldots, X_n;q,t,T]=\sum_{\lambda\in\cP} \frac{\prod_{i=1}^n \tilde H_\lambda[X_i;q,t] \prod_{i=1}^g N_{\lambda}(u_i;q,t)}{(\tilde H_\lambda, \tilde H_\lambda)^{S_{q,t}}} T^{|\lambda|},
\end{equation}
where
\[
N_\lambda(u;q,t) = \prod_{s\in \lambda} (q^{a(s)} - u t^{l(s)+1}) (q^{a(s)+1} - u^{-1} t^{l(s)}).
\]
For the purpose of this introduction, $X_i$ for each $i$ denotes an infinite set of variables $X_{i1}$, $X_{i2}$,$X_{i3}$, \ldots, and the function \eqref{eq:omegaintro} is symmetric in each set. We find it more convenient to think about $X_i$ as a formal symbol representing a generator of a free $\lambda$-ring, and this approach will be used throughout the paper. 
The summation runs over the set of all partitions, and for each partition $\lambda$ we have a product of several terms. For each puncture we take the modified Macdonald polynomial $\tilde H_\lambda[X_i;q,t]$ of \cite{garsia1999explicit}. It is a symmetric function in $X_{i1}$, $X_{i2}$,\ldots, whose coefficients are polynomials in $q$ and $t$  with integer coefficients. If $g>0$ then for each $i=1,\ldots,g$ we have an extra variable $u_i$, and an extra factor $N_\lambda(u_i;q,t)$, defined as a product over the cells of $\lambda$. It is a polynomial in $q$ and $t$ and a Laurent polynomial in $u_i$. The denominator is the modified Hall inner product with respect to the \emph{modifier} $S_{q,t}:=-(q-1)(t-1)$. This means that the scalar product is given in the power sum basis by
\begin{equation}\label{eq:intro hall}
(p_\lambda, p_\mu)^{S_{q,t}} = (p_\lambda, p_\mu) \prod_i S_{q^{\lambda_i},t^{\lambda_i}} = (p_\lambda, p_\mu) \prod_i -(q^{\lambda_i}-1)(t^{\lambda_i}-1),
\end{equation}
where $(p_\lambda, p_\mu)$ is the usual Hall inner product \cite{macdonald1995symmetric}. We have
\[
(\tilde H_\lambda, \tilde H_\lambda)^{S_{q,t}} = N_\lambda(1;q,t) \in \Z[q,t].
\]
Finally, the remaining term $T^{|\lambda|}$ is  introduced to keep track of the degrees, and is necessary for convergence if $n=0$.

Note that in the case $g=0$ we do not have the variables $u_1,\ldots,u_g$ and the factors $N_\lambda(u_i;q,t)$, so the situation is simpler.

The \emph{logarithmic HLV kernel} is the unique series $\HH_{u_1,u_2,\ldots, u_g}[X_1,X_2,\ldots, X_n;q,t,T]$ satisfying
\begin{equation}\label{eq:log}
\Omega_{u_1,u_2,\ldots, u_g}[X_1,X_2,\ldots, X_n;q,t,T] \;=\; \pExp\left[-\frac{\HH_{u_1,u_2,\ldots, u_g}[X_1,X_2,\ldots, X_n;q,t,T]}{(q-1)(t-1)}\right].
\end{equation}
A traditional way to define $\pExp$ is to set 
\[
\pExp\left[\sum_{i=1}^\infty c_i T_i\right] = \prod_{i=1}^\infty \frac{1}{(1-T_i)^{c_i}}
\]
for a sum of monomials $T_i$ in some variables, and for constants $c_i$. Then to make sense of \eqref{eq:log} we expand both sides as Laurent power series in $q$ or $t$. A more canonical way to define $\pExp$ is in the setting of $\lambda$-rings, see below.

The conjectures, formulated in \cite{hausel2011arithmetic}, \cite{mozgovoy2012solutions}, \cite{carlsson_vertex_2016} and other works can be split into two parts.
\begin{conjecture}[Part 1]
Let $g,n\geq 0$ and let $\lambda=(\lambda^{(1)}, \lambda^{(2)},\ldots,\lambda^{(n)})$ be a tuple of partitions of same size $N$.
Denote by $\HH_{u_1,u_2,\ldots,u_g, \lambda}(q,t)$ the coefficient of $\HH_{u_1,\ldots,u_g}[X_1,\ldots,X_n;q,t,T]$ in front of the monomial $\prod_{i=1}^n\prod_j X_{ij}^{\lambda^{(i)}_j} T^N$. Then
\[
\HH_{u_1,u_2,\ldots,u_g, \lambda}(q,t)\in \Z[q,t,u_1,\ldots,u_g,u_1^{-1},\ldots,u_g^{-1}].
\]
\end{conjecture}

From the definition it is not hard to deduce that $\HH_{u_1,\ldots,u_g}[X_1,\ldots,X_n;q,t,T]$ is a Laurent polynomial in $u_1,\ldots,u_n$ whose coefficients are \emph{rational functions} in $q$, $t$. So Part 1 of the Conjecture predicts that the denominators of these rational functions are trivial.

The second part of the Conjecture interprets coefficients of these polynomials as some cohomological invariants associated to character varieties or moduli spaces of Higgs bundles. Suppose $X$ is a smooth affine variety over $\C$ of dimension $d$. Consider for each $i=0,\ldots,d$ the cohomology $H^i(X,\C)$. By a construction of Deligne \cite{deligne1971hodge2}, we have a canonical weight filtration on $H^i(X,\C)$:
\[
0=W_{i-1} H^i(X,\C) \subset W_i H^i(X,\C) \subset \cdots W_{2i} H^i(X,\C)=H^i(X,\C).
\]
It is convenient to package the dimensions of the steps in the weight filtration into the \emph{mixed Hodge polynomial} as follows:
\[
 \sum_{i=0}^d \sum_{j=0}^i (-1)^i q^{\frac{d-i}{2}} t^{\frac{j}{2}} \dim \left(W_{i+j} H^i(X,\C) / W_{i+j-1} H^i(X,\C)\right)\in\Z[q^{\frac12}, t^{\frac12}].
\]
In the case of character varieties we have (\cite{hausel2011arithmetic})
\begin{conjecture}[Part 2]
Setting $u_1=u_2=\ldots=u_g=(qt)^{-\frac12}$ in $\HH_{u_1,u_2,\ldots,u_g, \lambda}(q,t)$ we obtain the mixed Hodge polynomial of the character variety of a Riemann surface of genus $g$ with $n$ punctures and generic semi-simple conjugacy classes of type $\lambda$.
\end{conjecture}

The main result of this paper is a proof of Part 1 of the Conjecture (Corollary \ref{cor:main}). The proof is combinatorial. There are two main techniques developed in this paper, which we outline now.

\subsection{Admissibility}
The first technique centers around the notion of admissibility, similar to the one of \cite{kontsevich2010soibelman}. A general formulation is postponed until Section \ref{sec:modifiers}, but we present a simplified version to illustrate the main idea here. Let $C$ be a power series in some set of variables whose coefficients are rational functions in $q$ and $t$. Suppose the constant term of $C$ is $1$. Then we can write $C$ in the form
\[
C=\pExp\left[-\frac{L}{(q-1)(t-1)}\right],
\]
where $L$ is a power series with constant term $0$ whose coefficients are rational functions in $q$ and $t$. 
\begin{defn}
We call $C$ \emph{admissible} if the constant term of $C$ is $1$ and the coefficients of $L$ are polynomials in $q$ and $t$.
\end{defn}

We prove that the admissibility is closed under the following operation. Suppose $C[X,Y,Z]$ is a power series that depends as a symmetric function on two sets of variables $X$, $Y$, and additionally on some other variables, which we denote by $Z$. Expand $C$ as follows:
\[
C[X,Y,Z] = 1+\sum_i F_i[X] G_i[Y] H_i[Z],
\]
where $F_i$, $G_i$ are symmetric functions. Each function $H_i[Z]$ can be, for instance, a power series in several sets of variables whose coefficients are Laurent polynomials in $u_1,\ldots,u_g$ whose coefficients are rational functions in $q,t$. Consider ``the integral''
\begin{equation}\label{intC FGH}
\int^S_X C[X, X^*, Z]:=1+\sum_i (F_i, G_i)^S H_i[Z],
\end{equation}
where $(F_i, G_i)^S$ is defined in \eqref{eq:intro hall} for $S=S_{q,t}=-(1-q)(1-t)$, or, more generally, for an arbitrary \emph{good modifier} $S$ (Definition \ref{defn:goodmodifier}) in Section \ref{sec:modifiers}. An infinite sum of power series does not always produce a well-defined power series. So we assume that for each $d$ only finitely many $H_i[Z]$ contain monomials of degree $\leq d$. Then only finitely many terms contribute to any given coefficient of \eqref{intC FGH} and the sum does make sense. Suppose that the constant term of each $H_i[Z]$ is $0$. Then the constant term of $\eqref{intC FGH}$ is $1$. We show that if $C[X,Y,Z]$ is admissible and $S$ is a good modifier, then \eqref{intC FGH} is also admissible. For details, see Lemma \ref{lem:composition} and Theorems \ref{thm:composition}, \ref{thm:trace}. The proof is constructive in the following sense. Write
\[
C[X, Y, Z]=\pExp\left[-\frac{L[X,Y, Z]}{(q-1)(t-1)}\right],
\]
\[
\int^S_X C[X, X^*,Z] = \pExp\left[-\frac{L_{\int}[Z]}{(q-1)(t-1)}\right].
\]
We give a formula for the coefficients of $L_{\int}[Z]$ in terms of the coefficients of $L[X,Y,Z]$ as a certain sum over graphs, similar to the Feynman diagram decompositions in physics. Only finitely many graphs contribute to any given coefficient of $L_{\int}{[Z]}$, and we verify that the contribution of each graph is a polynomial, so the result is a  polynomial in $q$ and $t$. This formula is established first for the unmodified Hall inner product in a simpler case, corresponding to the situation when $C[X,Y,Z]$ can be factored as $C_1[X,Z]C_2[Y,Z]$ (Theorem \ref{thm:logconvolution}). Then we show how the arguments need to be modified in the general case (Theorem \ref{thm:logconvolution2}), and analyse how introduction of the modifier $S$ affects the result in Lemma \ref{lem:composition}.

An important special case of the above picture arises as follows:
\begin{defn}
Let $D$ be a linear operator on the space of symmetric functions in one set of variables over the ring of rational functions over $q,t$.
The \emph{kernel} of $D$ is the following power series in two sets of variables $X,Y$:
\[
K_D^S = D \pExp\left[\frac{XY}{S}\right],
\]
where we apply $D$ to functions in $X$. An operator $D$ is called admissible if its kernel $K_D^S$ is admissible 
\end{defn}
From the Cauchy formula \cite{macdonald1995symmetric}
\[
\pExp\left[\frac{XY}{S}\right] = \sum_{\lambda\in\cP} \frac{p_\lambda[X] p_\lambda[Y]}{(p_\lambda, p_\lambda)^S},
\]
it is easily deduced that if $F$ is a symmetric function, then $DF$ can be written as follows:
\[
(DF)[X] = \int_{Y}^S K_D^S[X, Y] F[Y^*].
\]
Therefore, if we apply an admissible operator to an admissible function, we obtain an admissible function.

\subsection{Admissibility of $\nabla$}
The second main ingredient is the proof of $S_{q,t}$-admissibility of our version of the operator $\nabla$ of \cite{bergeron1999science}, \cite{bergeron1999identities}, which is obtained from the original one by a simple sign change \eqref{eq:nabladef}. Admissibility of $\nabla$ is deduced from admissibility of the operator $\Delta_v$, also introduced in \cite{bergeron1999identities}. We prove admissibility of $\Delta_v$ (Theorem \ref{thm:admdelta}) by showing that $L_v[X, Y]$ defined by
\[
\Delta_v \pExp\left[-\frac{XY}{(q-1)(t-1)}\right] = \pExp\left[-\frac{L_v[X,Y]}{(q-1)(t-1)}\right]
\]
can be computed by a certain recursion, which follows from a recently established identity satisfied by $\Delta_v$ (see \cite{garsia2016five}).

Finally, we show that the operation of integration \eqref{intC FGH} allows us to ``build $\Omega$'' from $\nabla$ in a sequence of steps. Thus our main goal, which is the admissibility of $\Omega_{u_1,\ldots,u_g}[X_1,\ldots,X_n;q,t,T]$, is deduced from admissibility of $\nabla$ and the fact that each step preserves admissibility (Section \ref{sec:HLV}).

\begin{rem}
We finish this introduction by stating a simple special case of a fact, crucial in the proof of Lemma \ref{lem:composition} and Theorems \ref{thm:composition}, \ref{thm:trace}, which we find interesting on its own. Let $\Gamma=(V,E)$ be a connected graph with Betti number $b(\Gamma)=1 + \# E - \# V$. Let $m:V\cup E\to\Z_{>0}$ be a coloring of the vertices and edges of $E$ by positive integers. Suppose $m(v)|m(e)$ for each edge $e\in E$ incident to a vertex $v\in V$. In the case when $\Gamma$ is a tree we further assume that the g.c.d. of the numbers $m(e)$ ($e\in E$), $m(v)$ ($v\in V$) is $1$. Then we have
\[
(q-1) \frac{\prod_{e\in E} (q^{m(e)}-1)}{\prod_{v\in V} (q^{m(v)}-1)} \in (q-1)^{b(\Gamma)} \Z[q].
\]
\end{rem}

I thank Fernando Rodriguez Villegas and Erik Carlsson for useful discussions and the anonymous referees for suggestions on improving the exposition. This research was performed during my stay at SISSA, and partly at ICTP, Trieste, to which I am grateful for support and stimulating environment.

\section{Notations}
The ring of symmetric functions in infinitely many variables over $\Q$ is denoted as $\Sym$. A \emph{$\lambda$-ring} is a ring $\Lambda$ with an action of $\Sym$
\[
F[G]\quad (F\in\Sym,\; G\in \Lambda),
\]
 such that the power sums $p_n$ act by ring endomorphisms for all $n\in\Z_{>0}$ and 
\[
p_1[G]=G,\quad p_m[p_n[G]]=p_{mn}[G]\quad(G\in \Lambda,\; m,n\in\Z_{>0}).
\]
Free $\lambda$-rings with generators $X_1, X_2, \ldots$ will be denoted as $\Sym[X_1,X_2,\ldots]$. Note that $\Sym[X]$ and $\Sym$ is the same thing. We denote by 
$\Sym[X_1,X_2,\ldots]^+$ the ideal consisting of symmetric functions with no constant term, by $\Sym[[X_1,X_2,\ldots]]$ the completion of $\Sym[X_1,X_2,\ldots]$ with respect to this ideal, and by $\Sym[[X_1,X_2,\ldots]]^+$ the corresponding ideal in $\Sym[[X_1,X_2,\ldots]]$.

The set of partitions is denoted by $\cP$. The set of non-empty partitions is denoted by $\cP^*$. For any $\lambda\in\cP$ we write 
\[
\lambda=(\lambda_1, \lambda_2, \ldots, \lambda_{l(\lambda)}) \quad
(\lambda_1\geq \lambda_2 \geq \ldots \geq \lambda_{l(\lambda)}>0,\; l(\lambda)\geq 0).
\]
The \emph{size} of $\lambda$ is
\[
|\lambda| = \sum_{i=1}^{l(\lambda)} \lambda_i,
\]
the set of all partitions of size $n$ is
\[
\cP_n = \{\lambda\in\cP:\, |\lambda|=n\}.
\]
Denote
\[
\Aut(\lambda) = \{\pi\in S_{l(\lambda)}:\; \lambda_{\pi(i)}=\lambda_i\;\text{for all $1\leq i\leq l(\lambda)$}\}.
\]
For a partition $\lambda$ and a constant $c$ we denote
\[
c\lambda = (c \lambda_1, c \lambda_2,\ldots,c\lambda_{l(\lambda)}).
\]

In general, we will often assume our $\lambda$-ring $\Lambda$ is complete with respect to some decreasing filtration $\Lambda=J^0\supset J^1\supset\cdots$ satisfying $J^i J^{i'}\subset J^{i+i'}$ for each $i,i'\geq 0$. We usually do not specify the filtration, so that the exposition is less cluttered, and hope that it is obvious from the context. Tensor product of such rings means \emph{completed tensor product}. We call a  sequence $A_1, A_2, \ldots \in\Lambda$ \emph{well-behaved} if $p_\lambda[A_i]$ tends to $0$ when $i,|\lambda|$ tend to $\infty$. Equivalently, for each $k$ there exists $i_0$ such that $p_\lambda[A_i]\in J^k$ if $i>i_0$ or $|\lambda|>i_0$. A single element $X$ is well-behaved if the sequence $X, 0, 0,\ldots$ is. A series is well-behaved if the sequence of its coefficients is.

We often consider operators from $\Lambda\otimes \Sym[X]$ to $\Lambda\otimes \Sym[[X]]$. In such a situation we say that an operator is \emph{finite} if its image is in $\Lambda\otimes \Sym[X]$, and we say it is \emph{continuous} if it extends to a continuous operator from $\Lambda\otimes \Sym[[X]]$ to $\Lambda\otimes \Sym[[X]]$. Note that two such operators $U$, $V$ can be composed to form a new operator $U V$ if $V$ is finite or $U$ is continuous. We then say that $U$ and $V$ are \emph{composable}.

\emph{Plethystic exponential} is the following expression, well-defined if $G$ is well-behaved:
\[
\pExp[G] = \sum_{n=0}^\infty h_n[G] = \exp\left( \sum_{n=1}^\infty \frac{p_n[G]}n\right),
\]
where $h_n$ denotes the complete homogeneous symmetric function, and $p_n$ denotes the power sum of degree $n$. The inverse operation is given as follows:
\[
\pLog[1+G] = \sum_{n=1}^\infty \frac{\mu(n)}{n} p_n\left[\log(1+G)\right],
\]
where $\mu$ is the M\"obius function.

We also use $e_n$ for the elementary symmetric function of degree $n$, $m_\lambda$ for the monomial symmetric function, and set
\[
p_\lambda = \prod_{i=1}^{l(\lambda)} p_{\lambda_i}, \;h_\lambda = \prod_{i=1}^{l(\lambda)} h_{\lambda_i},\;e_\lambda = \prod_{i=1}^{l(\lambda)} e_{\lambda_i} \quad(\lambda\in\cP).
\]
\section{Convolution formula}\label{sec:convolution}
We consider the free $\lambda$-ring $\Lambda_{A,B}$ over $\Q$ with generators
\[
\{A_\lambda, B_\lambda\}_{\lambda\in\cP^*},
\]
bi-graded by putting $A_\lambda$, $B_\lambda$ in degrees $(|\lambda|, 0)$ and $(0, |\lambda|)$ respectively.
Consider the following two series in the appropriate completion of $\Lambda_{\BA,\BB}[X]$, the free $\lambda$-ring generated by $X$ and the generators of $\Lambda_{\BA,\BB}$:
\begin{equation}\label{eq:defnAB}
\BA[X] := \sum_{\lambda\in\cP^*} A_\lambda p_\lambda[X],\;
\BB[X] := \sum_{\lambda\in\cP^*} B_\lambda p_\lambda[X].
\end{equation}
We are interested in the Hall inner product of their plethystic exponentials with respect to the $X$ variable
\begin{equation}\label{eq:CAB}
C(\BA, \BB):= (\pExp[\BA[X]], \pExp[\BB[X]])_X = \sum_{n=0}^\infty C_n(\BA, \BB),
\end{equation}
where $C_n(\BA, \BB)\in \Lambda_{\BA,\BB}$ is homogeneous of bi-degree $(n,n)$ for each $n\geq 0$, $C_0(\BA,\BB)=1$. Note that the series $\BA[X]$ and $\BB[X]$ are well-behaved, so that the plethystic exponentials and the Hall inner product make sense. In fact, each $C_n(\BA,\BB)$ is an explicit polynomial in symbols of the form $p_k[A_\lambda]$ and $p_k[B_\lambda]$ for $k\in\Z_{>0}$, $\lambda\in\cP$.  The problem we are going to solve is to compute $L(\BA,\BB)$
\[
L(\BA, \BB) = \sum_{n=1}^\infty L_n(\BA, \BB),
\]
with each $L_n(\BA, \BB)\in \Lambda_{\BA,\BB}$ homogeneous of bi-degree $(n,n)$, such that
\begin{equation}\label{eq:LAB}
C(\BA, \BB) = \pExp[L(\BA, \BB)].
\end{equation}
This will clearly solve the problem of computing the Hall scalar product of arbitrary two well-behaved series of the form $\pExp[\cdots]$, and representing the result again in the same form.

\subsection{Expansion of $\pExp$ into types.}
\begin{defn}\label{defn:type}
A \emph{type}\footnote{We borrow the word ``type'' from \cite{hausel2011arithmetic}, is seems to be an abbreviation of ``$\GL_n$-type''. These classify combinatorial types of conjugacy classes and representations of $\GL_n$ over a finite field.} is an unordered tuple of pairs of the form $(k, \lambda)$, where $\lambda\in\cP^*$, $k\in\Z_{>0}$. The set of all types is denoted by $\cT$. More explicitly, choose a total ordering $\leq$ on the set of all pairs $(k, \lambda)$ as above. Then a type is given by a sequence
\[
\tau=\{(k_i, \lambda^{(i)})\}_{i=1}^r\quad ((k_1, \lambda^{(1)}) \leq (k_2, \lambda^{(2)}) \leq \cdots \leq (k_r, \lambda^{(r)}),\; r\in\Z_{\geq 0}).
\]
 The \emph{size} of a type $\tau=\{(k_i, \lambda^{(i)})\}_{i=1}^r$ is the sum
\[
|\tau| = \sum_{i=1}^r k_i |\lambda^{(i)}|.
\] 
The set of all types of size $n$ is denoted by $\cT_n$. The group of automorphisms of $\tau$ is defined as 
\[
\Aut(\tau) = \{\pi\in S_r:\, k_{\pi(i)} = k_i,\; \lambda^{(\pi(i))} = \lambda^{(i)}\;\text{for all $1\leq i \leq r$}\}.
\]
The partition $\opflat(\tau)$ is defined as the partition with components $k_i \lambda^{(i)}_j$, where $i, j$ run over $i=1,2,\ldots, r$, $j=1,2,\ldots, l(\lambda^{(i)})$.
\end{defn} 
For any type $\tau=\{(k_i, \lambda^{(i)})\}_{i=1}^r$ and a collection of elements $\BZ=\{Z_\lambda\}_{\lambda\in\cP^*}$ in some $\lambda$-ring we denote
\[
\BZ^\tau = \prod_{i=1}^r p_{k_i}[Z_{\lambda^{(i)}}].
\]
If $X$ is a single element of a $\lambda$-ring, we set $X_\lambda=p_\lambda[X]$ and define $X^\tau$ as above. We have
\[
X^\tau = \prod_{i=1}^r p_{k_i}[p_{\lambda^{(i)}}[X]] = p_{\opflat(\tau)}[X].
\]

The element $\BA^{\tau} \BB^{\tau'}$ for $\tau, \tau'\in\cT$ has bi-degree $(|\tau|, |\tau'|)$, and these elements form a basis of $\Lambda_{\BA,\BB}$. The expressions $L(\BA, \BB)$ and $C(\BA, \BB)$ can be expanded as follows:
\[
C(\BA, \BB) = \sum_{\tau, \tau' \in\cT} C_{\tau, \tau'} \BA^\tau \BB^{\tau'},\quad L(\BA, \BB) = \sum_{\tau, \tau' \in\cT} L_{\tau, \tau'} \BA^\tau \BB^{\tau'},
\]
\[
C_{\tau, \tau'}, L_{\tau, \tau'}\in \Q.
\]

First we show how to expand $\pExp$ in the following
\begin{prop}\label{prop:Exp_expansion}
We have
\[
\pExp\left[\sum_{\lambda\in\cP^*} A_\lambda \right] = \sum_{\tau\in\cT} g_\tau \BA^\tau,
\]
where for each type $\tau=\{(k_i, \lambda^{(i)})\}_{i=1}^r$ 
\[
g_\tau  = \frac{1}{\#\Aut(\tau) \; \prod_{i=1}^r k_i}. 
\]
\end{prop}
\begin{proof}
Simply use the expansion of $\pExp$ into the power sums.
\end{proof}

We compute $C_{\tau', \tau'}$ first.
\begin{prop}\label{prop:exptypes}
Let $\tau=\{(k_i, \lambda^{(i)})\}_{i=1}^r$, $\tau'=\{(k_i', \lambda'^{(i)})\}_{i=1}^{r'}$ be types with $|\tau|=|\tau'|$. The number $C_{\tau, \tau'}$ is given by
\[
C_{\tau, \tau'} = \begin{cases}
\frac{\#\Aut(\opflat(\tau)) \prod_{i=1}^{r} \prod_{j=1}^{l(\lambda^{(i)})} k_i \lambda^{(i)}_j }{\#\Aut(\tau) \#\Aut(\tau') \; \prod_{i=1}^r k_i \; \prod_{i=1}^{r'} k_i'} & \text{if $\opflat(\tau)=\opflat(\tau')$,}\\
0 & \text{otherwise.}
\end{cases}
\]
\end{prop}
\begin{proof}
Using Proposition \ref{prop:Exp_expansion} we write 
\[
\pExp[\BA[X]] = \sum_{\tau\in\cT} g_\tau \BA^\tau p_{\opflat(\tau)}[X],
\]
and similarly for $\BB[X]$. Then use the formula
\[
(p_\lambda, p_{\lambda'}) = \begin{cases}
\#\Aut(\lambda) \prod_{i=1}^{l(\lambda)} \lambda_i & \text{if $\lambda=\lambda'$,}\\
0 & \text{otherwise.}
\end{cases}
\]
\end{proof}

\subsection{Graph interpretation}
Suppose $\tau=\{(k_i, \lambda^{(i)})\}_{i=1}^r$, $\tau'=\{(k_i', \lambda'^{(i)})\}_{i=1}^{r'}$ are such that $\opflat(\tau)=\opflat(\tau')$. Note that $\#\Aut(\opflat(\tau))$ is the number of bijections $\varphi$ from the set 
\[
W=\{(i,j): 1\leq i\leq r,\; 1\leq j \leq  l(\lambda^{(i)})\}
\]
to the set
\[
W'=\{(i,j): 1\leq i\leq r',\; 1\leq j \leq  l(\lambda'^{(i)})\}
\]
satisfying $k_i \lambda_j^{(i)} = k'_{i'} \lambda'^{(i')}_{j'}$ whenever $\varphi(i,j) = (i',j')$. We represent such a bijection by a bipartite graph with multi-edges as follows. The vertex set is the disjoint union of two sets $V=V_1\sqcup V_2$ where $V_1=\{1,2,\ldots,r\}$ and $V_2=\{1,2,\ldots,r'\}$,
i.e. vertices correspond to the pairs $(k_i, \lambda^{(i)})$ in $\tau$ and $\tau'$. Vertices are colored by positive integers. The colors of $i\in V_1$ resp. $i'\in V_2$ are given by $k_i$ resp. $k_{i'}'$. The edges are also colored. We connect $i\in V_1$ to $i'\in V_2$ by an edge with color $k_i \lambda_j^{(i)}$ for each pair $j,j'$ such that $\varphi(i,j)=(i',j')$. 
The result is a bipartite graph with multi-edges $(V, E, m)$, where $m(v)$ resp. $m(e)$ denotes the color of vertex $v\in V$ resp. edge $e\in E$. The graph and the coloring satisfy the following condition:
\begin{defn}
A colored bipartite graph $(V, E, m)$ with coloring $m:V\cup E\to \Z_{>0}$ is called \emph{admissible} if for each edge $e$ adjacent to a vertex $v$ we have $m(v)| m(e)$. We write admissible$^*$ to denote the additional requirement that the graph has no isolated vertices.
\end{defn}
Note that in our construction all admissible graphs without isolated vertices appear.

For example, consider two types $\tau=((2, (3)), (3, (2,2)))$ and $\tau'=((2, (3,3)), (3,(2)))$. We have $\opflat(\tau)=\opflat(\tau')=(6,6,6)$. There are $\#\Aut(\opflat(\tau))=6$ bijections. We have $W=((1,1), (2,1), (2,2))$ and $W'=((1,1), (1,2), (2,1))$. For instance, for the bijection $(1,1)\to (1,1)$, $(2,1)\to (1,2)$, $(2,2)\to (2,1)$ we obtain the following graph with $V_1=\{1,2\}$, $V_2=\{1,2\}$:
\[
\begin{tikzpicture}[scale=2,baseline=-.1cm] 
\draw (0,0.5)node[draw,circle, inner sep=4pt,fill=white] {$2$}  to node[midway,above] {$6$} (1.3,0.5);
\draw (0,-0.5) to  node[midway,above] {$6$} (1.3,0.5) node[draw,circle, inner sep=4pt,fill=white] {$2$};
\draw (0,-0.5) node[draw,circle, inner sep=4pt,fill=white] {$3$} to  node[midway,above] {$6$} (1.3,-0.5) node[draw,circle, inner sep=4pt,fill=white] {$3$};
\draw (0,0.9) node {$V_1$};
\draw (1.3,0.9) node {$V_2$};
\end{tikzpicture}
\]
There are other bijections producing the same graph, for instance $(1,1)\to (1,2)$, $(2,1)\to (1,1)$, $(2,2)\to (2,1)$. On the other hand, the bijection $(1,1)\to (2,1)$, $(2,1)\to (1,1)$, $(2,2)\to (1,2)$ produces a different graph:
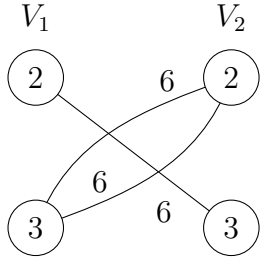
\[
\begin{tikzpicture}[scale=2,baseline=-.1cm] 
\draw (0,0.5)node[draw,circle, inner sep=4pt,fill=white] {$2$}  to node[near end,below left] {$6$} (1.3,-0.5) node[draw,circle, inner sep=4pt,fill=white] {$3$};
\draw (0,-0.5) to[out=70,in= -160] node[near end,above] {$6$} (1.3,0.5);
\draw (0,-0.5)node[draw,circle, inner sep=4pt,fill=white] {$3$} to[out=20,in=-110] node[near start,above] {$6$} (1.3,0.5) node[draw,circle, inner sep=4pt,fill=white] {$2$};
\draw (0,0.9) node {$V_1$};
\draw (1.3,0.9) node {$V_2$};
\end{tikzpicture}
\]

For an admissible bipartite graph $(V,E,m)$ without isolated vertices we obtain all triples $(\tau, \tau', \varphi)$ that give rise to $(V,E,m)$ as follows. To each $v\in V_1$ we associate a pair $(m(v), \lambda^{(v)})$, where $\lambda^{(v)}=\frac{1}{m(v)} \lambda_E^{(v)}$, and $\lambda_E^{(v)}$ is the partition with parts $m(e)$, where $e$ runs over all edges adjacent to $v$. The collection of these pairs forms $\tau$. Similarly we construct $\tau'$. To describe possible bijections $\varphi$ we use the edges. There is an ambiguity we need to fix. Each edge $e$ connecting $i$ and $i'$ creates a part $\frac{m(e)}{m(i)}$ for the partition corresponding to $i$ in $\tau$, and a part $\frac{m(e)}{m(i')}$ for the partition corresponding to $i'$ in $\tau'$, and $\varphi$ should associate the two parts. When we have several copies of the same part in $\lambda^{(i)}$, we have to further choose which of these parts will be associated to which part at the other end of $e$. We notice that the product of the groups 
\[
\Aut(V):=\prod_{i\in V_1} \Aut(\lambda^{(i)}) \times \prod_{i\in V_2} \Aut(\lambda'^{(i)}) = \prod_{v\in V} \Aut(\lambda^{(v)})
\]
transitively acts on all these choices with the stabilizer
\[
\Aut(E):=\prod_{i\in V_1, i'\in V_2} \Aut(\lambda^{(i,i')}) \subset \Aut(V),
\]
where $\lambda^{(i,i')}\in \cP$ is the partition consisting of the colors of all the edges between $i$ and $i'$. 
Thus each graph corresponds to $\frac{\#\Aut(V)}{\#\Aut(E)}$ bijections $\varphi$.
Hence we obtain 
\begin{prop}\label{prop:Cgraphs}
The number $C_{\tau, \tau'}$ is given as:
\[
C_{\tau, \tau'} = \sum_{E, m} \frac{\#\Aut(V)\; \prod_{e\in E} m(e)}{\#\Aut(E)\; \#\Aut(\tau)\;   \#\Aut(\tau') \prod_{v\in V} m(v)},
\]
where the sum is running over all $E$, $m$ producing the pairs $(k_i, \lambda^{(i)})$ resp.
$(k_i', \lambda'^{(i)})$ at vertices $i\in V_1$ resp. $i\in V_2$.
\end{prop}

\subsection{Decomposition into connected components}
For any bipartite admissible graph $\Gamma=(V, E, m)$ with $V=V_1\sqcup V_2$ define its \emph{weight} as
\[
w(\Gamma)=w(\Gamma)(\BA, \BB) = \frac{\#\Aut(V)\; \prod_{e\in E} m(e)}{\prod_{v\in V} m(v)} \BA^\tau \BB^{\tau'} \in \Lambda_{\BA,\BB}.
\]
It can also be written as
\begin{equation}
w(\Gamma) = \frac{1}{\prod_{x\in V\cup E} m(x)} \prod_{v\in V}
\left(p_{m(v)}[\text{$\BA[X]$ or $\BB[X]$}], p_{\lambda_E^{(v)}}[X]\right)_X,
\end{equation}
where we take $\BA[X]$ for $v\in V_1$ and $\BB[X]$ for $v\in V_2$.
The weight is multiplicative:
\[
w(\Gamma_1 \sqcup \Gamma_2) = w(\Gamma_1) w(\Gamma_2),
\]
and we have
\[
C(\BA, \BB) = \sum_{r,r'=0}^\infty\quad \sum_{\substack{\Gamma=(V, E, m)\;\text{admissible$^*$,} \\ V_1=\{1,\ldots,r\},\, V_2=\{1,\ldots,r'\}}} \quad \frac{w(\Gamma)(\BA,\BB)}{\#\Aut(E)\; r!\, r'!}.
\]
The last identity is true because in the right hand side each pair of types $\tau$, $\tau'$ appears $\frac{r!\, r'!}{\#\Aut(\tau)\,\#\Aut(\tau')}$ times.

Finally, we forget the identifications $V_1=\{1,\ldots,r\}$, $V_2=\{1,\ldots,r'\}$ and obtain
\begin{equation}
C(A, B) = \sum_{\Gamma\;\text{admissible$^*$}} \frac{w(\Gamma)(\BA,\BB)}{\#\Aut(\Gamma)}.
\end{equation}
In this formula the automorphisms are understood as permutations of vertices and edges, preserving the incidence relation, the marking and the decomposition $V=V_1\sqcup V_2$.

Each graph $\Gamma$ is a union of its connected components in a unique way\footnote{By our convention the empty graph is not connected}. Let $\Gamma$ be the union of $m_i$ copies of $\Gamma_i$ for $i=1,2,\ldots, k$. The automorphism group of $\Gamma$ is the product
\[
\Aut(\Gamma) = \prod_{i=1}^k \Aut(\Gamma_k)^{m_i} \rtimes S_{m_i}.
\]
This leads to
\begin{equation}\label{eq:conncomp}
C(\BA, \BB) = \exp\left(\sum_{\substack{\Gamma\;\text{admissible$^*$,}\\\text{connected}}} \frac{w(\Gamma)(\BA,\BB)}{\#\Aut(\Gamma)} \right).
\end{equation}

\subsection{Scaling operations}\label{sec:scaling}
Finally, we analyse what happens when we apply the following \emph{scaling} operation. For each $n\in\Z_{>0}$ and $\Gamma$ as above let $p_n(\Gamma)$ be the graph with $V, E$ the same as in $\Gamma$, and all values of $m$ multiplied by $n$. For a connected graph $\Gamma$ we have
\begin{equation}\label{eq:weightscaling}
w(p_n(\Gamma)(\BA,\BB)) = p_n[w(\Gamma)(\BA,\BB)]\; n^{b(\Gamma)-1},
\end{equation}
where $b(\Gamma) = \# E - \# V + 1\geq 0$ is the first Betti number of $\Gamma$. We call a graph \emph{primitive} if it cannot be obtained as $p_n[\Gamma]$ for $n>1$, i.e. if the g.c.d. of all the values of $m$ is $1$. Then every graph is expressed as $p_n(\Gamma)$ for a unique $n$ and a primitive graph $\Gamma$. Thus we have
\[
\log C(\BA,\BB) = \sum_{\substack{\Gamma\;\text{admissible$^*$,}\\\text{connected, primitive}}}
\sum_{n=1}^\infty n^{b(\Gamma)-1} \frac{p_n[w(\Gamma)(\BA,\BB)]}{\#\Aut(\Gamma)}.
\]
The M\"oebius inversion formula for $\pLog$, the operation inverse to $\pExp$, is:
\[
\pLog[X] = \sum_{n=1}^\infty \frac{\mu(n)}n p_n[\log(X)].
\]
In our case, this leads to
\[
L(\BA,\BB) = \pLog C(\BA,\BB) = \sum_{\substack{\Gamma\;\text{admissible$^*$,}\\\text{connected, primitive}}}
\sum_{n=1}^\infty n^{b(\Gamma)-1} \prod_{p|n\;\text{prime}} (1-p^{-b(\Gamma)}) \, \frac{p_n[w(\Gamma)]}{\#\Aut(\Gamma)}.
\]
Finally, we pass back to a sum over all connected graphs using \eqref{eq:weightscaling}:
\begin{thm}[Logarithmic convolution]\label{thm:logconvolution}
We have
\begin{equation}\label{eq:logconvolution}
L(\BA,\BB) = \sum_{\substack{\Gamma\;\text{bipartite,}\\\text{admissible$^*$,}\\\text{connected}}}
\phi_{b(\Gamma)}(g(\Gamma)) \frac{w(\Gamma)(\BA,\BB)}{\#\Aut(\Gamma)},
\end{equation}
where $g(\Gamma)$ denotes the g.c.d. of all the values of the coloring, and 
\[
\phi_k(n)=\prod_{p|n\;\text{prime}} (1-p^{-k}).
\]
\end{thm}
\begin{rem}
Note that for $k=0$ and $n>1$ we have $\phi_k(n)=0$. Hence terms with $g(\Gamma)>1$ disappear when $\Gamma$ is a tree. In what follows we will consider only primitive trees.
\end{rem}
\begin{rem}\label{rem:allow constant terms}
In this Section we chose to work with series $A$ and $B$ without constant terms. This is reflected by the fact that in \eqref{eq:defnAB} the summation is over $\cP^*$. Correspondingly, in Definition \ref{defn:type} we have only non-empty partitions. This is convenient because the number of types of given size is finite, and correspondingly the expressions $C_n(\BA,\BB)$ in \eqref{eq:CAB} have finitely many terms, so we have $C_n(\BA,\BB)\in\Lambda_{\BA,\BB}$. This is the reason why in \eqref{eq:logconvolution} we have to exclude graphs with isolated vertices. There are exactly two such graphs (see \ref{sec:degree0}). If we allow $\BA$ and $\BB$ to have constant terms $A_{()}$, $B_{()}$ and replace $\Lambda_{\BA,\BB}$ by its degree completion with respect to $A_{()}, B_{()}$, then for the corresponding $L(\BA,\BB)$ defined by \eqref{eq:CAB}, \eqref{eq:LAB} the formula \eqref{eq:logconvolution} holds with ``admissible$^*$'' replaced by ``admissible''. So the result looks prettier, but requires more effort to rigorously formulate.
\end{rem}

\subsection{Examples}\label{sec:exgraphs}
We compute first few terms in the expansion \eqref{eq:logconvolution}. Define the \emph{degree} of an admissible graph as the sum.
\[
|\Gamma| := \sum_{e\in E} m(e),
\]
The simplest case is
\subsubsection{Degree $0$}\label{sec:degree0}
All admissible connected graphs of degree $0$ consist of isolated vertices, which we have excluded. As explained in Remark \ref{rem:allow constant terms}, it is convenient to
include the following two graphs in the formula \eqref{eq:logconvolution}\footnote{Drawing bipartite graphs we paint the vertices from $V_1$ in black and the ones from $V_2$ in white. We omit the values of $m$ when they are $1$}:
\[
\Gamma_1 = \begin{tikzpicture}[scale=3] 
\draw (0,0) node[draw,circle, inner sep=2pt,fill] {};
\end{tikzpicture}
\;,\;
\Gamma_2 = \begin{tikzpicture}[scale=3] 
\draw (0,0) node[draw,circle, inner sep=2pt,fill=white] {};
\end{tikzpicture}
\;.
\]
Now if we add to $\Lambda_{A,B}$ two more generators $A_{()}$, $B_{()}$ of degree $0$, and set $w(\Gamma_1)(\BA,\BB) = A_{()}$, $w(\Gamma_2)(\BA,\BB) = B_{()}$, the formula \eqref{eq:logconvolution} holds with ``admissible$^*$'' replaced by ``admissible''. So we have
\[
L(\BA,\BB) = A_{()} + B_{()} + \cdots.
\]
\subsubsection{Degree $1$}
We have only $1$ graph which is connected and has exactly one edge:
\[
\Gamma = 
\begin{tikzpicture}[scale=2] 
\draw (0,0) node[draw,circle, inner sep=2pt,fill] {} -- (1,0) node[draw,circle, inner sep=2pt,fill=white] {};
\end{tikzpicture}
\]
This graph has weight $w(\Gamma)=A_{(1)} B_{(1)}$.
\subsubsection{Degree $2$}
In degree $2$ we have the following graphs:
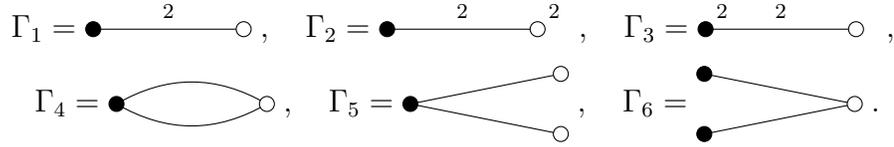
\[
\Gamma_1 = 
\begin{tikzpicture}[scale=2] 
\draw (0,0) node[draw,circle, inner sep=2pt,fill] {} -- node[above, midway]{\tiny 2} (1,0) node[draw,circle, inner sep=2pt,fill=white] {};
\end{tikzpicture}
\;,\quad
\Gamma_2 = 
\begin{tikzpicture}[scale=2] 
\draw (0,0) node[draw,circle, inner sep=2pt,fill] {} -- node[above, midway]{\tiny 2} (1,0) node[draw,circle, inner sep=2pt,fill=white] {} node[above right]{\tiny 2};
\end{tikzpicture}
\;,\quad
\Gamma_3 = 
\begin{tikzpicture}[scale=2] 
\draw (0,0) node[draw,circle, inner sep=2pt,fill]{} node[above right]{\tiny 2} -- node[above, midway]{\tiny 2} (1,0) node[draw,circle, inner sep=2pt,fill=white] {} node[above right]{};
\end{tikzpicture}
\;,\;
\]
\[
\Gamma_4 = 
\begin{tikzpicture}[scale=2,baseline=-.1cm] 
\draw (0,0) to [bend left=30] (1,0);
\draw (0,0) node[draw,circle, inner sep=2pt,fill] {} to [bend right=30] (1,0) node[draw,circle, inner sep=2pt,fill=white] {};
\end{tikzpicture}
\;,\quad
\Gamma_5 = 
\begin{tikzpicture}[scale=2,baseline=-.1cm] 
\draw (0,0) -- (1,0.2) node[draw,circle, inner sep=2pt,fill=white] {};
\draw (0,0) node[draw,circle, inner sep=2pt,fill] {} to (1,-0.2) node[draw,circle, inner sep=2pt,fill=white] {};
\end{tikzpicture}
\;,\quad
\Gamma_6 = 
\begin{tikzpicture}[scale=2,baseline=-.1cm] 
\draw (0,0.2) node[draw,circle, inner sep=2pt,fill] {} -- (1,0);
\draw (0,-0.2) node[draw,circle, inner sep=2pt,fill] {} to (1,0) node[draw,circle, inner sep=2pt,fill=white] {};
\end{tikzpicture}
\;.
\]
Their weights are given as follows:
\[
w(\Gamma_1) = 2 A_{(2)} B_{(2)},\;w(\Gamma_2) = A_{(2)} p_2[B_{(1)}],\;w(\Gamma_3) = p_2[A_{(1)}] B_{(2)},
\]
\[
w(\Gamma_4) = 4 A_{(1,1)} B_{(1,1)},\;w(\Gamma_5) = 2 A_{(1,1)} B_{(1)}^2,\;w(\Gamma_6) =  2 A_{(1)}^2 B_{(1,1)}.
\]
The orders of the automorphism groups are
\[
\#\Aut(\Gamma_i)=1\;(i=1,2,3),\quad \#\Aut(\Gamma_i)=2\;(i=4,5,6).
\]
So we obtain
\[
L(\BA,\BB) = A_{()} + B_{()} + A_{(1)} B_{(1)} + 2 A_{(2)} B_{(2)} + A_{(2)} p_2[B_{(1)}]
\]
\[ + p_2[A_{(1)}] B_{(2)} + 2 A_{(1,1)} B_{(1,1)} + A_{(1,1)} B_{(1)}^2 + A_{(1)}^2 B_{(1,1)}+\cdots.
\]
\subsubsection{Degree $3$}
In degree $3$ we have the following graphs:
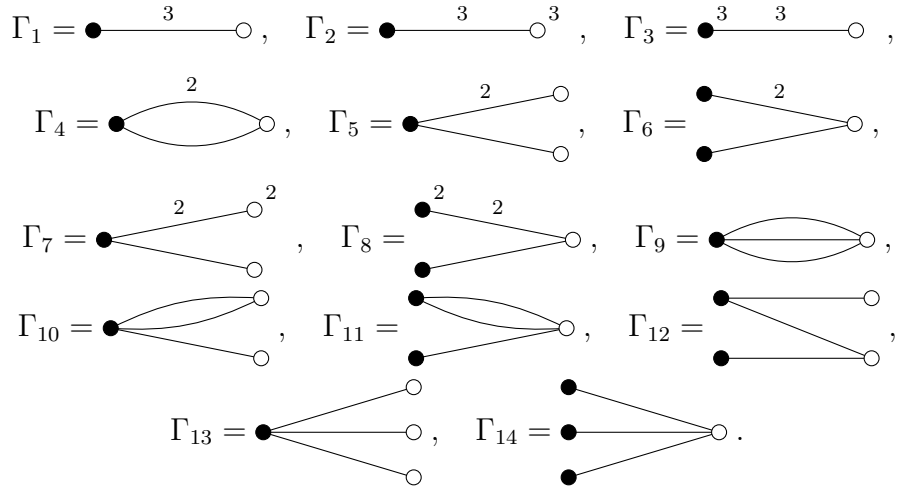
\[
\Gamma_1 = 
\begin{tikzpicture}[scale=2] 
\draw (0,0) node[draw,circle, inner sep=2pt,fill] {} -- node[above, midway]{\tiny 3} (1,0) node[draw,circle, inner sep=2pt,fill=white] {};
\end{tikzpicture}
\;,\quad
\Gamma_2 = 
\begin{tikzpicture}[scale=2] 
\draw (0,0) node[draw,circle, inner sep=2pt,fill] {} -- node[above, midway]{\tiny 3} (1,0) node[draw,circle, inner sep=2pt,fill=white] {} node[above right]{\tiny 3};
\end{tikzpicture}
\;,\quad
\Gamma_3 = 
\begin{tikzpicture}[scale=2] 
\draw (0,0) node[draw,circle, inner sep=2pt,fill]{} node[above right]{\tiny 3} -- node[above, midway]{\tiny 3} (1,0) node[draw,circle, inner sep=2pt,fill=white] {} node[above right]{};
\end{tikzpicture}
\;,\;
\]
\[
\Gamma_4 = 
\begin{tikzpicture}[scale=2,baseline=-.1cm] 
\draw (0,0) to [bend left=30] node[above, midway]{\tiny 2} (1,0);
\draw (0,0) node[draw,circle, inner sep=2pt,fill] {} to [bend right=30] (1,0) node[draw,circle, inner sep=2pt,fill=white] {};
\end{tikzpicture}
\;,\quad
\Gamma_5 = 
\begin{tikzpicture}[scale=2,baseline=-.1cm] 
\draw (0,0) -- node[above, midway]{\tiny 2} (1,0.2) node[draw,circle, inner sep=2pt,fill=white] {};
\draw (0,0) node[draw,circle, inner sep=2pt,fill] {} to (1,-0.2) node[draw,circle, inner sep=2pt,fill=white] {};
\end{tikzpicture}
\;,\quad
\Gamma_6 = 
\begin{tikzpicture}[scale=2,baseline=-.1cm] 
\draw (0,0.2) node[draw,circle, inner sep=2pt,fill] {} -- node[above, midway]{\tiny 2} (1,0);
\draw (0,-0.2) node[draw,circle, inner sep=2pt,fill] {} to (1,0) node[draw,circle, inner sep=2pt,fill=white] {};
\end{tikzpicture}
\;,
\]
\[
\Gamma_7 = 
\begin{tikzpicture}[scale=2,baseline=-.1cm] 
\draw (0,0) -- node[above, midway]{\tiny 2} (1,0.2) node[draw,circle, inner sep=2pt,fill=white] {}node[above right]{\tiny 2};
\draw (0,0) node[draw,circle, inner sep=2pt,fill] {} to (1,-0.2) node[draw,circle, inner sep=2pt,fill=white] {};
\end{tikzpicture}
\;,\quad
\Gamma_8 = 
\begin{tikzpicture}[scale=2,baseline=-.1cm] 
\draw (0,0.2) node[draw,circle, inner sep=2pt,fill] {}node[above right]{\tiny 2} -- node[above, midway]{\tiny 2} (1,0);
\draw (0,-0.2) node[draw,circle, inner sep=2pt,fill] {} to (1,0) node[draw,circle, inner sep=2pt,fill=white] {};
\end{tikzpicture}
\;,\quad
\Gamma_9 = 
\begin{tikzpicture}[scale=2,baseline=-.1cm] 
\draw (0,0) to [bend left=30] (1,0);
\draw (0,0) to  (1,0);
\draw (0,0) node[draw,circle, inner sep=2pt,fill] {} to [bend right=30] (1,0) node[draw,circle, inner sep=2pt,fill=white] {};
\end{tikzpicture}
\;,
\]
\[
\Gamma_{10} = 
\begin{tikzpicture}[scale=2,baseline=-.1cm] 
\draw (0,0) to[bend left=15] (1,0.2);
\draw (0,0) to[bend right=15] (1,0.2) node[draw,circle, inner sep=2pt,fill=white] {};
\draw (0,0) node[draw,circle, inner sep=2pt,fill] {} to (1,-0.2) node[draw,circle, inner sep=2pt,fill=white] {};
\end{tikzpicture}
\;,\quad
\Gamma_{11} = 
\begin{tikzpicture}[scale=2,baseline=-.1cm] 
\draw (0,0.2) to [bend left=15] (1,0);
\draw (0,0.2) node[draw,circle, inner sep=2pt,fill] {} to [bend right=15] (1,0);
\draw (0,-0.2) node[draw,circle, inner sep=2pt,fill] {} to (1,0) node[draw,circle, inner sep=2pt,fill=white] {};
\end{tikzpicture}
\;,\quad
\Gamma_{12} = 
\begin{tikzpicture}[scale=2,baseline=-.1cm] 
\draw (0,0.2) -- (1,0.2) node[draw,circle, inner sep=2pt,fill=white] {};
\draw (0,0.2)node[draw,circle, inner sep=2pt,fill] {}  -- (1,-0.2);
\draw (0,-0.2) node[draw,circle, inner sep=2pt,fill] {} to (1,-0.2) node[draw,circle, inner sep=2pt,fill=white] {};
\end{tikzpicture}
\;,
\]
\[
\Gamma_{13} = 
\begin{tikzpicture}[scale=2,baseline=-.1cm] 
\draw (0,0) -- (1,0.3) node[draw,circle, inner sep=2pt,fill=white] {};
\draw (0,0) -- (1,0) node[draw,circle, inner sep=2pt,fill=white] {};
\draw (0,0) node[draw,circle, inner sep=2pt,fill] {} to (1,-0.3) node[draw,circle, inner sep=2pt,fill=white] {};
\end{tikzpicture}
\;,\quad
\Gamma_{14} = 
\begin{tikzpicture}[scale=2,baseline=-.1cm] 
\draw (0,0.3) node[draw,circle, inner sep=2pt,fill] {} -- (1,0);
\draw (0,0) node[draw,circle, inner sep=2pt,fill] {} -- (1,0);
\draw (0,-0.3) node[draw,circle, inner sep=2pt,fill] {} to (1,0) node[draw,circle, inner sep=2pt,fill=white] {};
\end{tikzpicture}
\;.
\]
These graphs produce $14$ terms
\[
3 A_{(3)} B_{(3)} + A_{(3)} p_3[B_{(1)}] + p_3[A_{(1)}] B_{(3)} 
\]
\[
+ 2 A_{(2,1)} B_{(2,1)} + 2 A_{(2,1)} B_{(2)}B_{(1)} + 2 A_{(2)}A_{(1)}  B_{(2,1)}
\]
\[
+ A_{(2,1)} p_2[B_{(1)}] B_{(1)} + p_2[A_{(1)}] A_{(1)}  B_{(2,1)} + 6 A_{(1,1,1)} B_{(1,1,1)}
\]
\[
+ 6 A_{(1,1,1)} B_{(1,1)} B_{(1)} + 6 B_{(1,1,1)} A_{(1,1)} A_{(1)} + 4 A_{(1,1)} A_{(1)} B_{(1,1)} B_{(1)}
\]
\[
+ A_{(1,1,1)} B_{(1)}^3 + A_{(1)}^3 B_{(1,1,1)}.
\]

\section{Convolution formula II}\label{sec:convolution2}
To prove admissibility of the HLV kernel when genus greater than zero we will need a more general operation than the convolution of the previous section. The reader interested only in the genus zero case may safely proceed to Section \ref{sec:modifiers}. Consider the free $\lambda$-ring $\Lambda_\BA$ over $\Q$ with generators
\[
\{A_{\lambda, \mu}\}_{\lambda,\mu\in\cP}.
\]
Define
\[
\BA[X, X^*] = \sum_{\lambda,\mu\in\cP} A_{\lambda, \mu} p_\lambda[X] p_\mu[X^*].
\]
Define the exponential convolution as
\[
C(\BA) := \int_X \pExp[\BA[X, X^*]],
\]
where $\int_X$ is the Hall inner product map, viewed as a linear map
\[
\Sym[X, X^*] \cong \Sym[X]\otimes\Sym[X] \to \Q.
\]
In other words, it is the linear map $\Sym[X, X^*]\to \Q$ satisfying
\[
\int_X F[X] G[X^*] = (F[X], G[X])_X.
\]
Similarly to the previous Section, we define $L[\BA]$ by
\[
C(\BA) = \pExp[L(\BA)],
\]
and we would like to compute $L(\BA)$ as a sum over graphs. It turns out that all the constructions of the previous Section go through, and the result is essentially the same, except that the graphs we obtain are directed graphs (with multiedges and loops) instead of bipartite graphs. We go through the key points of the previous Section to highlight the differences.

\begin{itemize}
\item The constant term $A_{(),()}$ produces just a factor $\pExp[A_{(),()}]$ in front of $C(\BA)$, therefore it only adds $A_{(), ()}$ to the result $L(\BA)$. So we reduce to the case $A_{(),()}=0$. We can treat $A_{(), ()}$ as the contribution of the graph with one vertex and no edges  (the situation is analogous to Remark \ref{rem:allow constant terms}).
\item Instead of pairs $(d, \lambda)$ we have triples $(d, \lambda, \mu)$ with $d\in \Z_{>0}$, $\lambda, \mu\in\cP$ such that $(\lambda, \mu)\neq ((),())$. We choose a total ordering on such triples and define \emph{double types} as sequences 
\[
\tau=\{(k_i, \lambda^{(i)}, \mu^{(i)})\}_{i=1}^r:
\]
\[
\quad ((k_1, \lambda^{(1)}, \mu^{(1)}) \leq (k_2, \lambda^{(2)}, \mu^{(2)}) \leq \cdots \leq (k_r, \lambda^{(r)}, \mu^{(r)}),\; r\in\Z_{\geq 0}).
\]
\item For each double type as above we have two partitions $\opflat_1(\tau)$, $\opflat_2(\tau)$, obtained by taking all the parts of $k_i\lambda^{(i)}$ ($i=1,2,\ldots,r$), $k_i\mu^{(i)}$  ($i=1,2,\ldots,r$) respectively.
\item The notation $\BA^\tau$ means
\[
\BA^\tau = \prod_{i=1}^r p_{k_i}[A_{\lambda^{(i)},\mu^{(i)}}],
\]
and 
\[
X^\tau = \prod_{i=1}^r p_{k_i}[p_{\lambda^{(i)}}[X]] p_{k_i}[p_{\mu^{(i)}}[X^*]] = p_{\opflat_1(\tau)}[X] p_{\opflat_2(\tau)}[X^*],
\]
and an analog of Proposition \ref{prop:Exp_expansion} holds.
\item Analog of Proposition \ref{prop:exptypes} holds with
\[
C_{\tau} = \begin{cases}
\frac{\#\Aut(\opflat_1(\tau)) \prod_{i=1}^{r} \prod_{j=1}^{l(\lambda^{(i)})} k_i \lambda^{(i)}_j }{\#\Aut(\tau) \; \prod_{i=1}^r k_i } & \text{if $\opflat_1(\tau)=\opflat_2(\tau)$,}\\
0 & \text{otherwise.}
\end{cases}
\]
\item The bijections $\varphi$ are from the set 
\[
W=\{(i,j): 1\leq i\leq r,\; 1\leq j \leq  l(\lambda^{(i)})\}
\]
to the set
\[
W'=\{(i,j): 1\leq i\leq r,\; 1\leq j \leq  l(\mu^{(i)})\}.
\]
\item Instead of a colored bipartite graph we construct a colored directed graph with vertex set $V=\{1,2,\ldots,r\}$. Multiple edges and loops are allowed. The admissibility condition is the same (see Definition \ref{defn:dirgraphadmissible}).
\item For each vertex $v$ we associate two partitions: $\lambda^{(v)}=\frac{1}{m(v)} \lambda_E^{(v)}$, $\mu^{(v)}=\frac{1}{m(v)} \mu_E^{(v)}$, where $\lambda_E^{(v)}$ resp. $\mu_E^{(v)}$ contains all the colors of the outgoing resp. incoming edges of a vertex $v$. The triples $(m(v), \lambda^{(v)}, \mu^{(v)})$ form the type associated to $\Gamma$. For each pair $v,v'\in V$ the partition $\lambda^{(v,v')}$ contains colors of all the edges from $v$ to $v'$ .
\item We define
\[
\Aut(V) = \prod_{v\in V} \Aut(\lambda^{(v)}) \times \Aut(\mu^{(v)}),
\]
\[
\Aut(E) = \prod_{v,v'\in V} \Aut(\lambda^{(v,v')})\subset \Aut(V).
\]
Then an analog of Proposition \ref{prop:Cgraphs} holds:
\[
C_{\tau} = \sum_{E, m} \frac{\#\Aut(V) \prod_{e\in E} m(e)}{\#\Aut(E) \#\Aut(\tau) \prod_{v\in V} m(v)}.
\]
\item The weight is defined as
\[
w(\Gamma) = w(\Gamma)(\BA) = \frac{\#\Aut(V) \prod_{e\in E} m(e)}{\prod_{v\in V} m(v)} \prod_{v\in V} p_{m(v)}[A_{\lambda^{(v)}, \mu^{(v)}}],
\]
and we have
\[
C(\BA)=\sum_{\Gamma \;\text{admissible}}\; \frac{w(\Gamma)(\BA)}{\#\Aut(\Gamma)},
\]
where the automorphism group acts on edges and vertices, preserving the labels, the incidence relation, and the directions of edges.
\item Discussions in Section \ref{sec:scaling} go unchanged.
\end{itemize}

To summarize, we have

\begin{defn}\label{defn:dirgraphadmissible}
A colored directed graph $(V, E, m)$ with coloring $m:V\cup E\to \Z_{>0}$ is called \emph{admissible} if for each edge $e$ adjacent to a vertex $v$ we have $m(v)| m(e)$.
\end{defn}

\begin{thm}[Logarithmic convolution II]\label{thm:logconvolution2}
We have
\begin{equation}\label{eq:logconvolution2}
L(\BA) = \sum_{\substack{\Gamma\;\text{directed,}\\\text{admissible,}\\\text{connected}}}
\phi_{b(\Gamma)}(g(\Gamma)) \frac{w(\Gamma)(\BA)}{\#\Aut(\Gamma)},
\end{equation}
where $g(\Gamma)$ denotes the g.c.d. of all the values of the coloring, and 
\[
\phi_k(n)=\prod_{p|n\;\text{prime}} (1-p^{-k}).
\]
\end{thm}

\begin{example}
Let $\BA[X, X^*] = A X X^*$ over the free $\lambda$-ring generated by $A$. We can use the Theorem with $A_{(1),(1)}=A$ and all other terms $0$. Therefore, in the sum over graphs we only have connected graphs such that each vertex $v$ has exactly one outgoing and one incoming edge. Moreover, the colors of the edges and the vertices must be all equal, let's say $m(e)=m(v)=m\in\Z_{>0}$.  Such a graph is necessarily a cycle, say on $n$ vertices with $n=1,2,3,\ldots$. We obtain
\[
\pLog \int_X \pExp[A X X^*] = \sum_{n,m=1}^\infty \phi_1(m) \frac{p_m[A]^n}{n}.
\]
For instance, if $A=u$ is a monomial, i.e. $p_m[u]=u^m$ for all $m$, we have
\[
\pLog \int_X \pExp[u X X^*] = \sum_{n,m=1}^\infty \frac{\varphi(m)}{mn} u^{mn},
\]
where $\varphi(m)$ denotes the Euler's totient function. The last expression evaluates to 
\[
\pLog \int_X \pExp[u X X^*]= \sum_{n=1}^\infty u^n = \frac{u}{1-u}.
\]
\end{example}

\section{Modifiers}\label{sec:modifiers}
We analyze how introduction of certain modifiers in the Hall product influences \eqref{eq:logconvolution} and \eqref{eq:logconvolution2}.
\begin{defn}\label{defn:goodmodifier}
Let $\Lambda$ be a $\lambda$-ring. We call an element $S\in \Lambda$ a \emph{good modifier} if the following two properties hold:
\begin{enumerate}
\item For each $n\in\Z_{>0}$ the element $p_n[S]$ is not a zero divisor.
\item For each $n\in\Z_{>0}$ we have $S | p_n[S]$.
\item For each relatively prime pair $m,n\in\Z_{>0}$ we have $p_m[S] p_n[S] \,|\, S\, p_{mn}[S]$.
\end{enumerate}
\end{defn}

Typical modifiers are $S=q-1$, $S=1-q$ for $\Lambda=\Q[q]$,  and $S_{q,t}:=-(1-q)(1-t)$ for $\Lambda=\Q[q,t]$.

\begin{rem}
Let us show that $m|n$ implies $p_m[S]|p_n[S]$. Indeed, suppose $n=m k$ with $k\in\Z_{>0}$. By (ii) we have $p_k[S]=Q S$ for some $Q\in \Lambda$. Since $p_m$ is an algebra homomorphism and $p_m=p_m\circ p_k$, we obtain
\[
p_n[S]=p_m[p_k[S]] = p_m[QS] = p_m[Q] p_m[S].
\]
\end{rem}

Given a good modifier $S$ we introduce the \emph{modified Hall inner product} on $\Lambda\otimes\Sym[X]$ as follows:
\[
(F[X], G[X])_X^S = (F[SX], G[X])_X \quad (F,G\in\Lambda\otimes\Sym[X]).
\]
Note that because of the following identity the modified inner product is symmetric:
\[
(F[SX], G[X])_X = (F[X], G[SX])_X
\]
Denote by $\Lambda_S$ the localization
\[
\Lambda_S = \Lambda[S^{-1}, p_2[S]^{-1}, p_3[S]^{-1},\cdots].
\]
The reproducing kernel of the identity operator with respect to the modified inner product is 
\[
\pExp\left[\frac{XY}{S}\right]\in\Lambda_S\otimes\Sym[[X,Y]],
\]
i.e. the following identity holds
\[
\left(\pExp\left[\frac{XY}{S}\right], F[Y]\right)^S_Y = F[X] \quad (F\in\Lambda\otimes\Sym[X]).
\]
For any $\Lambda$-linear operator $U:\Lambda\otimes\Sym[X] \to \Lambda_S\otimes\Sym[[X]]$ we define its kernel as
\[
K^S_U[X,Y] = U \pExp\left[\frac{XY}{S}\right]\in\Lambda_S\otimes\Sym[[X,Y]].
\]
Then we have
\[
\left(K^S_U[X,Y], F[Y]\right)^S_Y = (UF)[X] \quad (F\in\Lambda\otimes\Sym[X]).
\]
For any two composable operators $U, V$ we have
\begin{equation}\label{eq:compositionkernel}
K^S_{UV}[X,Y] = \left(K^S_U[X,Z], K^S_V[Z, Y]\right)^S_Z.
\end{equation}
Now we define an important class of operators:
\begin{defn}
In general,
an expression $A\in\Lambda_S$ is called $S$-\emph{admissible} over $\Lambda$ if it is of the form $\pExp\left[\frac{L}{S}\right]$ for $L\in\Lambda$.
A $\Lambda$-linear operator $U:\Lambda\otimes\Sym[X] \to \Lambda_S\otimes\Sym[[X]]$ is called $S$-\emph{admissible} over $\Lambda$ if its kernel has the form
\[
K^S_{U}[X,Y] = \pExp\left[\frac{L^S_U[X,Y]}{S}\right]\;\text{with}\; L^S_U[X,Y]\in \Lambda\otimes\Sym[[X, Y]],
\]
where $L_U^S[0, 0]$, i.e. the constant term of $L_U^S[X, Y]$, is well-behaved.
\end{defn}

It turns out that sometimes admissibility implies \emph{integrality}:
\begin{prop}
If $U$ is an $S$-admissible operator such that $U(1)\in \Lambda\otimes \Sym[[X]]$, then for any $F\in\Lambda\otimes\Sym[X]$ we have
\[
UF\in\Lambda\otimes\Sym[[X]].
\]
\end{prop}
\begin{proof}
We have
\[
(UF)[X] = \left(\pExp\left[\frac{L^S_U[X,SY]}{S}\right], F[Y]\right)_Y.
\]
Write $L^S_U[X,Y] = L_0[X] + L_+[X, Y]$ where $L_0=L^S_U[X,0]$ and $L_+[X, Y]$ contains only terms of positive degree in $Y$. Then 
\[
\frac{L_+[X,SY]}{S}\in\Lambda\otimes\Sym[[X,Y]],\;\text{and} \quad U(1)=\pExp\left[\frac{L_0[X]}{S}\right] \in \Lambda\otimes\Sym[[X]]
\]
by the assumptions. Hence 
\[
\pExp\left[\frac{L^S_U[X,SY]}{S}\right] \in\Lambda\otimes\Sym[[X,Y]],
\]
which implies integrality of $UF$.
\end{proof}

Our main result about admissible operators is
\begin{thm}\label{thm:composition}
For any $S$-admissible composable operators $U$ and $V$ the composition $UV$ is also admissible.
\end{thm}

The proof follows from \eqref{eq:compositionkernel} and the following fact after replacing $\Lambda$ by $\Lambda[[X,Y]]$.
\begin{lem}\label{lem:composition}
For any $\lambda$-ring $\Lambda$ with a good modifier $S$, and for well-behaved $\BA, \BB\in\Lambda\otimes\Sym[[X]]$ we have
\[
\left(\pExp\left[\frac{\BA[X]}{S}\right], \pExp\left[\frac{\BB[X]}{S}\right]\right)_X^S=
\pExp\left[\frac{L^S(\BA, \BB)}{S}\right]\quad\text{with $L^S(\BA, \BB)\in\Lambda$.}
\]
\end{lem}
\begin{proof}
It is enough to prove the statement for $\Lambda_{\BA,\BB}$, the free $\lambda$-ring over $\Lambda$ generated by $\{A_\lambda, B_\lambda\}_{\lambda\in\cP}$, as in Section \ref{sec:convolution}, and $\BA[X]$, $\BB[X]$ as in \eqref{eq:defnAB}. See also Remark \ref{rem:allow constant terms} for the way to incorporate constant terms. Denote
\[
A_\lambda' = \frac{A_\lambda p_\lambda[S]} S,\quad B_\lambda' = \frac{B_\lambda}{S}.
\]
To compute $C$ we apply Theorem \ref{thm:logconvolution} for $\BA'$, $\BB'$:
\[
\frac{C}{S} = \sum_{\substack{\Gamma\;\text{admissible,}\\\text{connected}}}
\phi_{b(\Gamma)}(g(\Gamma)) \frac{w(\Gamma)(\BA', \BB')}{\#\Aut(\Gamma)}.
\]
It is enough to show that for each graph $\Gamma$ such that $\phi_{b(\Gamma)}(g(\Gamma))\neq 0$ we have
\[
S\, w(\Gamma)(\BA', \BB')\in\Lambda_{\BA,\BB}.
\]
By the definition of $w(\Gamma)$, $S w(\Gamma)(\BA', \BB')$ is a rational multiple of 
\[
S\, \frac{S^\tau}{\prod_{v\in V} p_{m(v)}[S]} \BA^\tau \BB^{\tau'}.
\]
Therefore, it is enough to show that the following expression is in $\Lambda$:
\[
c_\Gamma[S]:=S\, \frac{S^\tau}{\prod_{v\in V} p_{m(v)}[S]}
=S\, \frac{\prod_{e\in E} p_{m(e)}[S]}{\prod_{v\in V} p_{m(v)}[S]}.
\]
We show it by induction on the number of edges.
We have two cases: either $\Gamma$ is a primitive tree, or $\Gamma$ is not a tree.

Suppose $\Gamma$ is a primitive tree. Then removing a leaf vertex $v$ adjacent to an edge $e$ we obtain a tree $\Gamma'$, which is not necessarily primitive. Denote $g=g(\Gamma')$ and let $\Gamma'=p_g[\Gamma'']$, so that $\Gamma''$ is primitive. Then we have
\[
c_{\Gamma''}[S]\in \Lambda,\quad c_\Gamma[S] = S\, \frac{p_{m(e)}[S]}{p_{m(v)}[S]} p_g\left[\frac{c_{\Gamma''}[S]}{S}\right]=\frac{S\, p_{m(e)}[S]}{p_{m(v)}[S] p_g[S]} p_g[c_{\Gamma''}[S]].
\]
We have $\gcd(g, m(v))=1$ because $\Gamma$ is primitive, and $g|m(e)$, $m(v)|m(e)$ because $\Gamma$ is admissible. Therefore $g m(v) | m(e)$ and the statement follows from the axioms (ii) and (iii) of Definition \ref{defn:goodmodifier}.

Now suppose that $\Gamma$ has a cycle. Then it has an edge $e$ such that removing $e$ does not destroy connectivity. Let $\Gamma'$ be the graph obtained by removing $e$, $g=g(\Gamma')$, and $\Gamma'=p_g[\Gamma'']$. We have $c_{\Gamma''}\in \Lambda$ and
\[
c_{\Gamma}[S] = S\, p_{m(e)}[S] p_g\left[\frac{c_{\Gamma''}[S]}{S}\right] = S\,\frac{p_{m(e)}[S]}{p_g[S]} p_g[c_{\Gamma''}[S]].
\]
This belongs to $S\, \Lambda\subset \Lambda$ because $g|m(e)$ by admissibility of $\Gamma$.
\end{proof}

Moreover, the proof gives us the following statement:
\begin{cor}[of the proof]\label{cor:Slogconv}
We have the following explicit formula for $L^S[\BA,\BB]$:
\[
L^S[\BA,\BB] = \sum_{\substack{\Gamma\;\text{admissible,}\\\text{connected}}}
\phi_{b(\Gamma)}(g(\Gamma)) \frac{w(\Gamma)(\BA, \BB)}{\#\Aut(\Gamma)}  c_\Gamma[S],
\]
where for each $\Gamma=(V, E, m)$ with $\phi_{b(\Gamma)}(g(\Gamma))$ we have
\[
c_\Gamma[S]=S\, \frac{\prod_{e\in E} p_{m(e)}[S]}{\prod_{v\in V} p_{m(v)}[S]} \in S^{b(\Gamma)}\Lambda.
\]
\end{cor}

Similarly, we can employ the second convolution formula, in this case to compute the trace of an operator.
\begin{defn}
An $S$-admissible operator $U$ is said to be of \emph{trace class} if the sequence of the coefficients of $L_U^S[X,Y]$ is well-behaved. In this case we define its trace by
\[
\Tr U:=\int_X^S K_U^S[X, X^*],
\]
where $\int_X^S$ is defined analogously to $\int_X$, but using the $S$-modified inner product.
\end{defn}
We have
\begin{thm}\label{thm:trace}
The trace of an $S$-admissible operator $U$ of trace class is admissible.
\end{thm}

\section{Admissibility of $\Delta$ and $\nabla$}
Now we set $S=S_{q,t}=-(1-q)(1-t)$, which is a good modifier in $\Q[q,t]$. The modified Macdonald polynomials $\{\tilde H_{\lambda}\}_{\lambda\in\cP}$ form a basis of $\Q(q,t)\otimes\Sym[X]$. For each $F\in \Q[q,t]\otimes\Sym[X]$ define $\Delta_{F}:\Q[q,t]\otimes\Sym[X] \to \Q(q,t)\otimes\Sym[X]$ in the Macdonald basis as follows:
\[
\Delta_{F} \tilde H_{\lambda} = F[B_\lambda],\quad B_\lambda = \sum_{(c,r)\in\lambda} q^c t^r,
\]
the summation is over the cells $(c,r)$ of $\lambda$, where $c$ is the column index and $r$ is the row index.

Over the slightly bigger ring $\Q(q,t)[[u]]$ define the operator
\[
\Delta_u = \sum_{n=0}^\infty (-u)^n \Delta_{e_n}.
\]
Then we have
\[
\Delta_u^{-1} = \sum_{n=0}^\infty u^n \Delta_{h_n}.
\]
Notice that
\[
\Delta_u \tilde H_\lambda = \prod_{r,c\in\lambda} (1-u q^c t^r)\;\tilde H_\lambda \quad(\lambda\in\cP).
\]
Note that $\Delta_u$ restricted to the space of symmetric functions of degree $k$ is a polynomial in $u$ of degree $k$. Its degree $k$ term is denoted by $\nabla$ (see \cite{bergeron1999identities} for an overview of results about this operator, and note that our $\nabla$ is different from the original one by a sign $(-1)^{|\lambda|}$),
\begin{equation}\label{eq:nabladef}
\nabla \tilde H_\lambda = (-1)^{|\lambda|} q^{n'(\lambda)} t^{n(\lambda)}\; \tilde H_\lambda,
\end{equation}
where $n(\lambda)=\sum_{i=1}^{l(\lambda)} (i-1)\lambda_i$, $n'(\lambda)= \sum_{i=1}^{l(\lambda)} \binom{\lambda_i}{2}$.
We also need the shift operators $\tau_u$:
\[
(\tau_u F)[X] = F[X+u],\quad \tau:=\tau_1,
\]
and their $S_{q,t}$-conjugates
\[
(\tau_u^* F)[X] = F[X] \pExp\left[u\,\frac{X}{S_{q,t}}\right],\quad \tau^*:=\tau_1^*.
\]
Also we have the following partially defined operation on operators. For any continuous operator $U$ denote by $\cS^{-1}(U)$, if it exists, the unique continuous operator satisfying
\[
\tau^*\tau\, U = \cS^{-1}(U)\, \tau^* \tau.
\]
Note that $\tau$ has it's image in $\Q[q,t]\otimes\Sym[X]$, $\tau^*$ is continuous, and both operators are $S_{q,t}$-admissible. 

Denote $\Delta_u' = \pExp[-u/S_{q,t}] \Delta_u$. The following identity was established in \cite{garsia2016five}:
\begin{equation}\label{eq:fiveterm}
\Delta_{v}^{-1} \tau_u \Delta_{v} \tau_u^{-1} = \nabla^{-1} \tau_{uv} \nabla
= \cS^{-1}(\Delta'_{uv}),
\end{equation}
where the operators act on $\Q(q,t)[[u,v]]\otimes\Sym[X]$. This result was motivated by a conjecture in \cite{bergeron2013tableaux}. In \cite{bergeron2013tableaux} the identity \eqref{eq:fiveterm} is shown to imply certain generalized Pieri rules for Macdonald polynomials. The main idea is that in the basis of Macdonald polynomials the operators $\Delta_v$ and $\nabla$ are easily described, while the operators $\tau_u$ are difficult. So we write
\[
\Delta_v^{-1} \tau_u \Delta_v = \nabla^{-1}  \tau_{uv} \nabla \tau_u,
\]
and then recursively express the result of a single application of the ``difficult''  $\tau$-operator on the left using double application of $\tau$-operators on the right. Our situation is the opposite, because we want to understand $\pLog$ of the kernel of our operators, so the $\tau$-operators are ``easy'', while the $\nabla$ and $\Delta$-operators are ``difficult''. Thus we employ the other implication of \eqref{eq:fiveterm}:
\[
\tau_u \Delta_{v} \tau_u^{-1} = \Delta_{v} \cS^{-1}(\Delta_{uv}'),\quad\text{equivalent to}
\]
\begin{equation}\label{eq:fiveterm2}
\tau_u \Delta_{v} \tau_u^{-1}\tau^*\tau = \Delta_{v} \tau^*\tau \Delta_{uv} \pExp\left[-\frac{uv}S_{q,t}\right].
\end{equation}
Let $L_v[X, Y] = L_{\Delta_v}^{S_{q,t}}[X,Y]\in\Q(q,t)[[v]]\otimes\Sym[[X,Y]]^+$.
We compute the kernels of both sides of \eqref{eq:fiveterm2}:
\[
\tau_u \Delta_{v} \tau_u^{-1}\tau^*\tau \pExp\left[\frac{XY}{S_{q,t}}\right]
=\tau_u \Delta_v \pExp\left[\frac{(X-u)(Y+1)+Y}{S_{q,t}}\right]
\]
\[
=\tau_u \pExp\left[\frac{L_v[X, Y+1] + Y(1-u)-u}{S_{q,t}}\right]
\]
\[
=\pExp\left[\frac{L_v[X+u, Y+1] + Y(1-u)-u}{S_{q,t}}\right],
\]
\[
\Delta_{v} \tau^*\tau \Delta_{uv} \pExp\left[\frac{XY}{S_{q,t}}\right]
= \Delta_{v} \tau^*\tau \pExp\left[\frac{L_{uv}[X, Y]}{S_{q,t}}\right]
\]
\[
= \Delta_{v} \tau^* \pExp\left[\frac{L_{uv}[X+1, Y]}{S_{q,t}}\right]
\]
\[
= \left(\pExp\left[\frac{L_v[X, Z+1]}{S_{q,t}}\right], \pExp\left[\frac{L_{uv}[Z+1, Y]}{S_{q,t}}\right]\right)^{S_{q,t}}_Z,
\]
where in the last formula we used the kernel for the operator $\Delta_v \tau^*$.

Define elements $A_{v,\lambda}\in \Q(q,t)[[v]]\otimes \Sym[[X]]$ for $\lambda\in\cP^*$ by the formula 
\[
L_v[X, Z+1] = \sum_{\lambda\in\cP^*} A_{v,\lambda}[X] p_\lambda[Z].
\]
Since the operator $\Delta_v$ preserves the degree of symmetric functions,
we have that its kernel $K_{\Delta_v}^{S_{q,t}}[X,Y]$, and therefore also the logarithm $L_v[X,Y]$, is a sum of terms whose degree in $X$ equals to the degree in $Y$. So we can write
\[
L_v[X, Y] = \sum_{k=1}^\infty L^{(k)}_v [X, Y],
\]
where $L_k[X,Y]$ has degrees $k$ both in $X$ and in $Y$.
It follows that $A_{v,\lambda}[X]$ is a sum of terms whose degree in $X$ is at least $|\lambda|$.

Note that the operator $\Delta_v$ is self-adjoint with respect to the $S_{q,t}$-modified inner product because Macdonald polynomials are orthogonal. This implies that $K_{\Delta_v}^{S_{q,t}}[X,Y]$, and therefore also $L_v[X,Y]$, is symmetric: $L_v[X,Y]=L_v[Y,X]$. We use this to expand
\[
L_{uv}[Z+1,Y] = L_{uv}[Y,Z+1] = \sum_{\lambda\in\cP^*} A_{uv,\lambda}[Y] p_\lambda[Z].
\]
Then \eqref{eq:fiveterm2} and the kernel evaluations above imply:
\[
L_v[X+u, Y+1] + Y(1-u) - u = L^{S_{q,t}}(\BA_v[X], \BA_{uv}[Y])-uv,
\]
for $L^{S_{q,t}}$ as in Corollary \ref{cor:Slogconv}. Here $\BA_v[X]$ denotes the collection of elements $\{A_{v,\lambda}[X]\}_{\lambda\in\cP^*}$.
Equivalently,
\begin{equation}\label{eq:fivetermlog}
L_v[X+u, Y+1] = Y(u-1) + u(1-v) + L^{S_{q,t}}(\BA_v[X], \BA_{uv}[Y]).
\end{equation}

Note that the left hand side is a sum of terms with
\[
 (\text{degree in $X$} \;+\; \text{degree in $u$}) = k,\quad 
 \text{degree in $Y$} \leq k\quad (k\in\Z_{>0}).
\]
Denote by $T_k[X,Y;u,v]$ the sum of the terms on the right hand side of \eqref{eq:fivetermlog} with 
\begin{equation}\label{eq:term degrees}
 (\text{degree in $X$} \;+\; \text{degree in $u$}) = k,\quad \text{degree in $u$}\geq 1,\;
 \text{degree in $Y$} \leq k-1.
\end{equation}
Clearly these terms also have the degree in $X$ $\leq k-1$, therefore $T_k$ can be computed from $L_v^{(i)}$ with $i\leq k-1$, using only terms of Corollary \ref{cor:Slogconv} with $|\Gamma|\leq k-1$. On the other hand, we can write down the sum of the terms satisfying \eqref{eq:term degrees} on the left hand side of \eqref{eq:fivetermlog} as follows:
\[
T_k[X, Y; u, v] = L'^{(k)}_v[X+u, Y] - L'^{(k)}_v[X, Y],
\]
where
\[
L'^{(k)}_v[X, Y] = L^{(k)}_v[X, Y+1] - L^{(k)}_v[X, Y].
\]
Notice that $L^{(k)}_v[X, Y]$ is homogeneous of degree $k$ in $Y$, and $L'^{(k)}_v[X, Y]$ is still homogeneous of degree $k$ in $X$. By the following well-known fact we can recover $L'^{(k)}_v[X, Y]$ from $T_k[X, Y; 1,v]$, and then $L^{(k)}_v[X,Y]$ from $L'^{(k)}_v[X,Y]$:
\begin{lem}[Solving a symmetric recursion]\label{lem:symrec}
For any $k\geq 1$
the map from symmetric functions of degree $k$ to symmetric functions of degree $<k$ given by
\[
F[X] \to F'[X]:=F[X+1]-F[X]
\]
is injective.
\end{lem}
\begin{proof}
In the monomial basis, for $\lambda\in\cP_k$ we have
\[
m_\lambda[X+1]-m_\lambda[X] = \sum_{j=1}^r m_{\lambda_1,\lambda_2,\ldots, \widehat{\lambda_{i_j}},\ldots,\lambda_r}[X],
\]
where $i_1, i_2,\ldots,i_r$ are such that the sequence $\lambda_{i_1},\lambda_{i_2},\ldots, \lambda_{i_r}$ contains each element of the set of parts of $\lambda$ exactly once.  Thus, for instance, the map sending $m_{\lambda}$ to $m_{\lambda'}$ with $\lambda'=(k-|\lambda|, \lambda_1,\ldots,\lambda_{l(\lambda)})$ when $k-|\lambda|\geq \lambda_1$, and to $0$ otherwise is a left inverse for the map $F\to F'$.
\end{proof}

Thus we see that we can inductively compute each $L_v^{(k)}$ using only operations in the ring $\Q[q,t,v]$, which implies
\begin{thm}\label{thm:admdelta}
The operator $\Delta_v$ is $S_{q,t}$-admissible over $\Q[q,t,v]$.
\end{thm}

The coefficients of $L_\nabla^{S_{q,t}}$ can be extracted from the coefficients of $L_{\Delta_v}^{S_{q,t}}$. They are simply given by the terms of $L_{\Delta_v}^{S_{q,t}}$ whose degree in $v$ equals to the degree in $X$, and hence also equals to the degree in $Y$. So we have
\begin{cor}\label{cor:nabla}
The operator $\nabla$ is $S_{q,t}$-admissible over $\Q[q,t]$.
\end{cor}

Finally, the inverse of $\nabla$ is related to $\nabla$ by the identity
\[
\bar\omega \nabla^{-1} = \nabla \bar\omega,
\]
where $\bar\omega$ is the $\lambda$-ring automorphism of $\Q[q,t,q^{-1},t^{-1}]\otimes\Sym[X]$ which sends $q, t, X$ to $q^{-1}, t^{-1}, -X$. Thus we also have
\begin{cor}\label{cor:nablainv}
The operator $\nabla^{-1}$ is $S_{q,t}$-admissible over the ring $\Q[q,t,q^{-1},t^{-1}]$.
\end{cor}
In fact, the kernel of $\nabla^{-1}$ can be computed from the one of $\nabla$ as follows:
\[
L^{S_{q,t}}_{\nabla^{-1}}[X, Y; q, t] = q\, t\; L^{S_{q,t}}_{\nabla}[-X,\; -(qt)^{-1} Y;\; q^{-1},\; t^{-1}].
\]

\subsection{Example of computation}
We give the first few steps of the computation. So we start with $k=1$ and the initial approximation $L_v[X,Y]\approx 0$. This gives us the first value
\[
T_1[X, Y, u, v] = u(1-v).
\]
Solving the symmetric recursion for $k=1$, $T_1[X, Y, 1, v] = (1-v)$ gives
\[
L_v^{(1)}[X, Y] = X Y (1-v).
\]

Now we proceed to $k=2$, $L_v[X,Y]\approx X Y (1-v)$. We need to compute the operation $L^S$ for 
\[
L_v[X, Z+1] \;\approx\; X(Z+1)(1-v) = X(1-v) + X(1-v) Z,
\]
\[
 L_{uv}[Y, Z+1] \;\approx\; Y(1-uv) + Y(1-uv) Z.
\]
The graphs of degrees $0$ and $1$ (see Section \ref{sec:exgraphs}) produce
\[
X(1-v) + Y(1-uv) + X(1-v)Y(1-uv),
\]
which together with the extra summand $Y(u-1)+u(1-v)$ from \eqref{eq:fivetermlog} gives
\[
(1-v)((X+u)(Y+1)- uv X Y).
\]
Keeping only the terms satisfying the degree restrictions \eqref{eq:term degrees} we obtain
\[
T_2[X, Y;u, v] = u v (v-1) X Y,
\]
and the solution to the symmetric recursion is
\[
L_v^{(2)}[X, Y] =   v(v-1) e_2[X] e_2[Y].
\]

We will compute one more step $k=3$. Expansion of $L_v[X, Z+1]$ in the power sum basis is
\[
L_v[X, Z+1] \;\approx\; X(1-v) + X(1-v) Z + v(v-1) e_2[X]\left(\frac{Z^2-p_2[Z]}2 + p_1[Z]\right),
\]
\[
L_{uv}[Y, Z+1] \;\approx\; Y(1-uv) + Y(1-uv) Z + uv(uv-1) e_2[Y]\left(\frac{Z^2-p_2[Z]}2 + p_1[Z]\right).
\]
So the logarithmic convolution is 
\[
X(1-v) + Y(1-uv) + (X(1-v) + v(v-1) e_2[X])\,(Y(1-uv) + uv(uv-1) e_2[Y])
\]
\[
\frac12 v(v-1) e_2[X] uv(uv-1) e_2[Y] (q+1)(t+1) -\frac12 v(v-1) e_2[X] p_2[Y(1-uv)]
\]
\[
 -\frac12 p_2[X(1-v)]  uv(uv-1) e_2[Y] 
- \frac12 v(v-1) e_2[X] uv(uv-1) e_2[Y] (q-1)(t-1)
\]
\[
+\frac12 v(v-1) e_2[X] (Y(1-uv))^2 + \frac12 uv(uv-1) e_2[Y] (X(1-v))^2,
\]
where we kept only terms with degrees in $X$ and $Y$ not exceeding $2$. Keeping only the terms satisfying the degree restrictions \eqref{eq:term degrees} and setting $u=1$ we obtain
%\[
%v(v-1) e_2[X] (-uv e_2[Y] - uv Y) + X(1-v) u^2 v^2 e_2[Y] 
%\]
%\[
%-\frac12 v(v-1) e_2[X] \,uv\, e_2[Y] (q+1)(t+1) + \frac12 p_2[X(1-v)] u v e_2[Y]
%\]
%\[
%+\frac12 v(v-1) e_2[X] \,uv\, e_2[Y] (q-1)(t-1) - v(v-1) e_2[X] Y^2 u v -\frac12 u ve_2[Y] X^2(1-v)^2
%\]
%\[
%= e_2[X] e_2[Y] (-u v^2(v-1) (q+t+1)) + u v e_2[Y](-e_2[X] + v X^2 - v^2 X^2 + v^2 e_2[X]) - u v^2 (v-1) e_2[X] Y^2.
%\]
\[
T_2[X,Y;1,v]=v(v-1) (1-v(q+t)) e_2[X] e_2[Y] 
-v^2 (v-1) (e_2[Y] X^2 + e_2[X] Y^2) 
\]
\[
 - v^2(v-1)(X e_2[Y] + Y e_2[X]).
\]
Next, we pass to the monomial basis in order to apply the procedure from the proof of Lemma \ref{lem:symrec}. We throw away the terms $m_{(1,1)}[X] m_{(2)}[Y]$, because they will go to $0$ anyway. We are left with
\[
v(v-1) (1-v(q+t+4)) m_{(1,1)}[X] m_{(1,1)}[Y] - v^2(v-1) (X m_{(1,1)}[Y] + Y m_{(1,1)}[X]),
\]
So we find
\[
L_v^{(3)} = v(v-1) (1-v(q+t+4))\, m_{(1,1,1)}[X]\, m_{(1,1,1)}[Y]
\]
\[
- v^2(v-1) \left(m_{(2,1)}[X]\, m_{(1,1,1)}[Y] \,+\, m_{(1,1,1)}[X]\, m_{(2,1)}[Y]\right).
\]
Note that replacing $v$ by $-v$ makes all the coefficients positive. This is expected to hold in general.

\section{HLV kernels}\label{sec:HLV}
Fix integers $g,n\geq 0$. In what follows, the variables $q$, $t$, $T$, $u_i$ are monomial, i.e $p_k[q]=q^k$ for all $k$, and similarly for $t$, $T$ and $u_i$ for all $i$. The exponential HLV kernel of genus $g$ with $n$ punctures (see \cite{hausel2011arithmetic}, \cite{carlsson_vertex_2016}) is defined as 
\begin{equation}\label{eq:omega}
\Omega_{u_1,u_2,\ldots, u_g}[X_1,X_2,\ldots, X_n;q,t,T]=\sum_{\lambda\in\cP} \frac{\prod_{i=1}^n \tilde H_\lambda[X_i;q,t] \prod_{i=1}^g N_{\lambda}(u_i;q,t)}{(\tilde H_\lambda, \tilde H_\lambda)^S} T^{|\lambda|},
\end{equation}
where
\[
N_\lambda(u;q,t) = (-u)^{-|\lambda|} q^{n(\lambda')} t^{n(\lambda)}\prod_{s\in \lambda} (1-u q^{-a(s)} t^{l(s)+1}) (1-u t^{-l(s)} q^{a(s)+1}),
\]
the variable $s$ runs over the cells of $\lambda$, and $a(s)$, $l(s)$ denote the arm and the leg lengths of $s$ correspondingly. Note that for $n=0$ the variable $T$ is necessary for convergence. For $n>0$ we can set $T=1$ without loosing any information: the power of $T$ is always equal to the degree in any of the variables $X_1$, $X_2$,\ldots. The kernel belongs to the $\lambda$-ring
\[
\Q(q,t)[u_1,u_2,\ldots,u_g, u_1^{-1}, u_2^{-1},\ldots,u_g^{-1}][[T]]\otimes\Sym[[X_1,X_2,\ldots,X_n]].
\]

In \cite{carlsson_vertex_2016} a more convenient formula for $N_\lambda$ is given. We sketch a proof of this formula. It is convenient to use the notation (remember $S_{q,t}=-(1-q)(1-t)$)
\[
D_\lambda(q,t)=-1-S_{q,t} B_\lambda,\quad \bar D_\lambda(q,t) = D_\lambda(q^{-1}, t^{-1}).
\]
Then it is not hard to prove that:
\[
qt\frac{D_\lambda \bar D_\lambda-1}{S_{q,t}} = \sum_{s\in \lambda} q^{-a(s)} t^{l(s)+1} + t^{-l(s)} q^{a(s)+1},
\]
which implies
\[
N_\lambda(u;q,t) = (-u)^{-|\lambda|} q^{n(\lambda')} t^{n(\lambda)} \pExp\left[-qt\frac{D_\lambda \bar D_\lambda-1}{S_{q,t}}u\right] \in \Q[q,t,q^{-1},t^{-1}]((u)).
\]
We will use Tesler's operator (see \cite{garsia1999explicit}, \cite{mellit2016plethystic}), defined as $\nabla \tau^* \tau$, whose main property is:
\begin{equation}\label{eq:tesler}
\nabla \tau^* \tau \tilde H_\lambda = \pExp\left[\frac{D_\lambda X}{S_{q,t}}\right].
\end{equation}
Applying the operator $\bar\omega$ on both sides we obtain
\[
\nabla^{-1} \tau^{*}_{-qt} \tau_{-1} \nabla^{-1} \tilde H_\lambda = \pExp\left[-qt\frac{\bar D_\lambda X}{S_{q,t}}\right].
\]
We substitute $uX$ in the place of $X$ in the former expression, and take the $S$-modified scalar product with the latter expression:
\[
\left(\nabla \tau^*_u (\tilde H_\lambda[u X + 1]),\; \nabla^{-1} \tau^{*}_{-qt} \tau_{-1} \nabla^{-1} \tilde H_\lambda\right)^{S_{q,t}} = \pExp\left[-qt u\frac{D_\lambda \bar D_\lambda}{S_{q,t}}\right].
\]
Now we use the definition of $\nabla$ from \eqref{eq:nabladef} and the fact that the degree of $\tilde H_\lambda$ is $|\lambda|$:
\[
N_\lambda(u) = \left(\nabla \tau^*_u \tau_{u^{-1}} \tilde H_\lambda,\; \nabla^{-1} \tau^{*}_{-qt} \tau_{-1} \tilde H_\lambda\right)^{S_{q,t}} \pExp\left[\frac{uqt}{S_{q,t}}\right].
\]
Since $\nabla$ is self-adjoint with respect to the $S_{q,t}$-scalar product, we can remove it from both sides:
\[
N_\lambda(u) = \left(\tau^*_u \tau_{u^{-1}} \tilde H_\lambda,\; \tau^{*}_{-qt} \tau_{-1} \tilde H_\lambda\right)^{S_{q,t}} \pExp\left[\frac{uqt}{S_{q,t}}\right].
\]
Finally, we move $\tau^*_u$ to the right, where it becomes $\tau_u$, and pass it through $(\tau^{*}_{qt})^{-1}$, which gives an extra factor of $\pExp\left[-\frac{uqt}{S_{q,t}}\right]$. Then we move $(\tau^{*}_{qt})^{-1}$ to the left:
\begin{equation}\label{eq:N}
N_\lambda(u) = \left(\tau_{u^{-1}-qt}\tilde H_\lambda,\; \tau_{u-1} \tilde H_\lambda\right)^{S_{q,t}}.
\end{equation}
The identity is understood in the ring $\Q[q,t][u,u^{-1}]$. We are ready to prove our main result
\begin{thm}
For any $n,g\geq 0$ the exponential HLV kernel of genus $g$ with $n$ punctures is $S_{q,t}$-admissible over the ring
\[
\Lambda=\Q[q,t,u_1,u_2,\ldots,u_g, u_1^{-1}, u_2^{-1},\ldots,u_g^{-1}][[T]]\otimes\Sym[[X_1,X_2,\ldots,X_n]]
\]
with $S_{q,t}=-(q-1)(t-1)$. Equivalently, the logarithmic HLV kernel 
\[
\HH_{u_1,u_2,\ldots, u_g}[X_1,X_2,\ldots, X_n;q,t,T]
\]
\[
:=-(q-1)(t-1)\pLog \Omega_{u_1,u_2,\ldots, u_g}[X_1,X_2,\ldots, X_n;q,t,T]
\]
is in $\Lambda$.
\end{thm}
\begin{proof}
Each coefficient of $\Omega$ in the variables $u_i$, $T$ and $X_i$, as a function of $q$ and $t$ belongs to the intersection of the rings\footnote{The intersection is understood inside the ring $\Q(q,t)$}
\[
\Q(q,t) \cap \Q((q))[[t]] \cap \Q((t))[[q]].
\]
This follows from invertibility of the denominator in \eqref{eq:omega}, explicitly given as
\[
(\tilde H_\lambda, \tilde H_\lambda)^{S_{q,t}} = \prod_{s\in \lambda} (q^{a(s)}- t^{l(s)+1}) (q^{a(s)+1}- t^{l(s)} )
\]
in these three rings. Thus it is enough to show that $\Omega$ is admissible over
\[
\Lambda'=\Q[q,t,q^{-1}, t^{-1},u_1,u_2,\ldots,u_g, u_1^{-1}, u_2^{-1},\ldots,u_g^{-1}][[T]]\otimes\Sym[[X_1,X_2,\ldots,X_n]].
\]
Indeed, we have
\[
\Q(q,t) \cap \Q((q))[[t]] \cap \Q((t))[[q]] \cap\Q[q,t,q^{-1},t^{-1}] = \Q[q,t].
\]

The next step is to show how to reconstruct $\Omega$ using the operators $\nabla$, $\nabla^{-1}$, etc, which are $S_{q,t}$-admissible by Corollaries \ref{cor:nabla} and \ref{cor:nablainv}. Begin with the kernel of Tesler's operator $\nabla \tau^*\tau$, which is admissible by Corollary \ref{cor:nabla}. Applying $\nabla \tau^*\tau$ to the Cauchy kernel
\[
\pExp\left[\frac{XY}{S_{q,t}}\right] = \sum_{\lambda\in\cP} \frac{\tilde H_\lambda[X] \tilde H_\lambda[Y]}{(\tilde H_\lambda, \tilde H_\lambda)^{S_{q,t}}}
\]
we obtain
\[
 \sum_{\lambda\in\cP} \frac{\pExp\left[\frac{X D_\lambda}{S_{q,t}}\right] \tilde H_\lambda[Y]}{(\tilde H_\lambda, \tilde H_\lambda)^{S_{q,t}}}.
\]
So this expression is admissible. Now we perform the substitution $X=X_1+X_2+\cdots+X_{n+2g}$, $Y=T$. This does not affect admissibility and the result is
\[
 \sum_{\lambda\in\cP} \frac{\prod_{i=1}^{n+2g} \pExp\left[\frac{X_i D_\lambda}{S_{q,t}}\right]}{(\tilde H_\lambda, \tilde H_\lambda)^{S_{q,t}}} T^{|\lambda|}.
\]
Next we apply $\tau^*_{-1} \nabla^{-1}$. Note that $\nabla^{-1}$ is not admissible over $\Lambda$, but is still admissible over $\Lambda'$. By Lemma \ref{lem:composition} applying $\nabla^{-1}$ does not affect admissibility, and by \eqref{eq:tesler} we obtain
\[
 \sum_{\lambda\in\cP} \frac{\prod_{i=1}^{n+2g} \tilde H_\lambda[X_i+1] }{(\tilde H_\lambda, \tilde H_\lambda)^{S_{q,t}}} T^{|\lambda|}.
\]
Next we throw away all the terms whose degree in any of $X_1$, $X_2$, \ldots is less than the degree in $T$. This produces
\[
 \sum_{\lambda\in\cP} \frac{\prod_{i=1}^{n+2g} \tilde H_\lambda[X_i] }{(\tilde H_\lambda, \tilde H_\lambda)^{S_{q,t}}} T^{|\lambda|},
\]
precisely the kernel for $n+2g$ punctures. Next we apply $\tau_{u_i^{-1}-qt}$ in $X_{n+2i-1}$ and $\tau_{u_i-1}$ in $X_{n+2i}$ for each $i=1,2,\ldots,g$. Then we view the resulting expression as a kernel of an operator with $X=X_{n+2i-1}$, $Y=X_{n+2i}$ for $i=1, 2, \ldots, g$, and take the trace. By Theorem \ref{thm:trace}, the result is still admissible. We obtain
\[
 \sum_{\lambda\in\cP} \frac{\prod_{i=1}^{n} \tilde H_\lambda[X_i]  \prod_{i=1}^g (\tilde H_\lambda[X+u_i^{-1}-qt], \tilde H_\lambda[X +u_i-1])_{X}^{S_{q,t}} }{(\tilde H_\lambda, \tilde H_\lambda)^{S_{q,t}}} T^{|\lambda|},
\]
which by \eqref{eq:N} equals $\Omega_{u_1,u_2,\ldots,u_g}[X_1,X_2,\ldots,X_n]$.
\end{proof}

The coefficients of the expansion of  $\Omega_{u_1,u_2,\ldots,u_g}[X_1,X_2,\ldots,X_n]$ in the monomial basis as functions of $q$ and $t$ are in $\Z((q))[[t]]$, which follows from integrality of Macdonald polynomials. It is well-known that $\pLog$, when computed in the monomial basis has coefficients in $\Z$. This implies that the coefficients of $\HH$ in the monomial basis are in $\Z((q))[[t]]$, so by our result the coefficients are in $\Q[q,t]\cap \Z((q))[[t]]=\Z[q,t]$.

\begin{cor}\label{cor:main}
The coefficients of 
\[
\HH_{u_1,u_2,\ldots, u_g}[X_1,X_2,\ldots, X_n;q,t,T]
\]
in the monomial basis are polynomials in $q, t, u_i, u_i^{-1}$ with integer coefficients.
\end{cor}
\printbibliography

%\bibliography{refs.bib}{}

\end{document}